\documentclass[oneside,reqno,11pt]{amsart}

\usepackage{fullpage}

\usepackage{etoolbox}

\makeatletter
\let\old@tocline\@tocline
\let\part@tocline\@tocline
\let\section@tocline\@tocline
\newcommand{\subsection@dotsep}{4.5}
\newcommand{\subsubsection@dotsep}{4.5}
\patchcmd{\@tocline}
  {\hfil}
  {\nobreak
     \leaders\hbox{$\m@th
        \mkern \subsection@dotsep mu\hbox{.}\mkern \subsection@dotsep mu$}\hfill
     \nobreak}{}{}
\let\subsection@tocline\@tocline
\let\@tocline\old@tocline

\patchcmd{\@tocline}
  {\hfil}
  {\nobreak
     \leaders\hbox{$\m@th
        \mkern \subsubsection@dotsep mu\hbox{.}\mkern \subsubsection@dotsep mu$}\hfill
     \nobreak}{}{}
\let\subsubsection@tocline\@tocline
\let\@tocline\old@tocline

\let\old@l@subsection\l@subsection
\let\old@l@subsubsection\l@subsubsection

\def\@tocwriteb#1#2#3{%
  \begingroup
    \@xp\def\csname #2@tocline\endcsname##1##2##3##4##5##6{%
      \ifnum##1>\c@tocdepth
      \else \sbox\z@{##5\let\indentlabel\@tochangmeasure##6}\fi}%
    \csname l@#2\endcsname{#1{\csname#2name\endcsname}{\@secnumber}{}}%
  \endgroup
  \addcontentsline{toc}{#2}%
    {\protect#1{\csname#2name\endcsname}{\@secnumber}{#3}}}%

\newlength{\@tocpartindent}
\newlength{\@tocsectionindent}
\newlength{\@tocsubsectionindent}
\newlength{\@tocsubsubsectionindent}
\newlength{\@tocpartnumwidth}
\newlength{\@tocsectionnumwidth}
\newlength{\@tocsubsectionnumwidth}
\newlength{\@tocsubsubsectionnumwidth}
\newcommand{\settocpartnumwidth}[1]{\setlength{\@tocpartnumwidth}{#1}}
\newcommand{\settocsectionnumwidth}[1]{\setlength{\@tocsectionnumwidth}{#1}}
\newcommand{\settocsubsectionnumwidth}[1]{\setlength{\@tocsubsectionnumwidth}{#1}}
\newcommand{\settocsubsubsectionnumwidth}[1]{\setlength{\@tocsubsubsectionnumwidth}{#1}}
\newcommand{\settocpartindent}[1]{\setlength{\@tocpartindent}{#1}}
\newcommand{\settocsectionindent}[1]{\setlength{\@tocsectionindent}{#1}}
\newcommand{\settocsubsectionindent}[1]{\setlength{\@tocsubsectionindent}{#1}}
\newcommand{\settocsubsubsectionindent}[1]{\setlength{\@tocsubsubsectionindent}{#1}}

\renewcommand{\l@part}{\part@tocline{-1}{\@tocpartvskip}{\@tocpartindent}{}{\@tocpartformat}}%
\renewcommand{\l@section}{\section@tocline{1}{\@tocsectionvskip}{\@tocsectionindent}{}{\@tocsectionformat}}%
\renewcommand{\l@subsection}{\subsection@tocline{2}{\@tocsubsectionvskip}{\@tocsubsectionindent}{}{\@tocsubsectionformat}}%
\renewcommand{\l@subsubsection}{\subsubsection@tocline{3}{\@tocsubsubsectionvskip}{\@tocsubsubsectionindent}{}{\@tocsubsubsectionformat}}%
\newcommand{\@tocpartformat}{}
\newcommand{\@tocsectionformat}{}
\newcommand{\@tocsubsectionformat}{}
\newcommand{\@tocsubsubsectionformat}{}
\expandafter\def\csname toc@0format\endcsname{\@tocpartformat}
\expandafter\def\csname toc@1format\endcsname{\@tocsectionformat}
\expandafter\def\csname toc@2format\endcsname{\@tocsubsectionformat}
\expandafter\def\csname toc@3format\endcsname{\@tocsubsubsectionformat}
\newcommand{\settocpartformat}[1]{\renewcommand{\@tocpartformat}{#1}}
\newcommand{\settocsectionformat}[1]{\renewcommand{\@tocsectionformat}{#1}}
\newcommand{\settocsubsectionformat}[1]{\renewcommand{\@tocsubsectionformat}{#1}}
\newcommand{\settocsubsubsectionformat}[1]{\renewcommand{\@tocsubsubsectionformat}{#1}}
\newlength{\@tocpartvskip}
\newcommand{\settocpartvskip}[1]{\setlength{\@tocpartvskip}{#1}}
\newlength{\@tocsectionvskip}
\newcommand{\settocsectionvskip}[1]{\setlength{\@tocsectionvskip}{#1}}
\newlength{\@tocsubsectionvskip}
\newcommand{\settocsubsectionvskip}[1]{\setlength{\@tocsubsectionvskip}{#1}}
\newlength{\@tocsubsubsectionvskip}
\newcommand{\settocsubsubsectionvskip}[1]{\setlength{\@tocsubsubsectionvskip}{#1}}

\patchcmd{\tocsection}{\indentlabel}{\makebox[\@tocsectionnumwidth][l]}{}{}
\patchcmd{\tocsubsection}{\indentlabel}{\makebox[\@tocsubsectionnumwidth][l]}{}{}
\patchcmd{\tocsubsubsection}{\indentlabel}{\makebox[\@tocsubsubsectionnumwidth][l]}{}{}

\newcommand{\@sectypepnumformat}{}
\renewcommand{\contentsline}[1]{%
  \expandafter\let\expandafter\@sectypepnumformat\csname @toc#1pnumformat\endcsname%
  \csname l@#1\endcsname}
\newcommand{\@tocpartpnumformat}{}
\newcommand{\@tocsectionpnumformat}{}
\newcommand{\@tocsubsectionpnumformat}{}
\newcommand{\@tocsubsubsectionpnumformat}{}
\newcommand{\setpartpnumformat}[1]{\renewcommand{\@tocpartpnumformat}{#1}}
\newcommand{\setsectionpnumformat}[1]{\renewcommand{\@tocsectionpnumformat}{#1}}
\newcommand{\setsubsectionpnumformat}[1]{\renewcommand{\@tocsubsectionpnumformat}{#1}}
\newcommand{\setsubsubsectionpnumformat}[1]{\renewcommand{\@tocsubsubsectionpnumformat}{#1}}
\renewcommand{\@tocpagenum}[1]{%
  \hfill {\mdseries\@sectypepnumformat #1}}

\let\oldappendix\appendix
\renewcommand{\appendix}{%
  \leavevmode\oldappendix%
  \addtocontents{toc}{%
    \protect\settowidth{\protect\@tocsectionnumwidth}{\protect\@tocsectionformat\sectionname\space}%
    \protect\addtolength{\protect\@tocsectionnumwidth}{2em}}%
}
\makeatother

\makeatletter
\settocsectionnumwidth{2em}
\settocsubsectionnumwidth{2.5em}
\settocsubsubsectionnumwidth{3em}
\settocsectionindent{2em}%
\settocsubsectionindent{\dimexpr\@tocsectionindent+\@tocsectionnumwidth}%
\settocsubsubsectionindent{\dimexpr\@tocsubsectionindent+\@tocsubsectionnumwidth}%
\makeatother

\settocpartvskip{10pt}
\settocsectionvskip{10pt}
\settocsubsectionvskip{0pt}
\settocsubsubsectionvskip{0pt}
    
\settocpartformat{\bfseries}
\settocsectionformat{\bfseries}
\settocsubsectionformat{\mdseries}
\settocsubsubsectionformat{\mdseries}
\setpartpnumformat{\bfseries}
\setsectionpnumformat{\bfseries}
\setsubsectionpnumformat{\mdseries}
\setsubsubsectionpnumformat{\mdseries}

\let\oldtableofcontents\tableofcontents
\renewcommand{\tableofcontents}{%
  \vspace*{-\linespacing}
  \oldtableofcontents}

\usepackage{import}
\usepackage{url}
\usepackage[table]{xcolor}
    \definecolor{RC-purple}{rgb}{0.85,0,1}

\usepackage[shortlabels]{enumitem}
\usepackage{amsmath}
\usepackage{amssymb}
\usepackage{amsthm}
\usepackage{mathtools}
\usepackage{mathrsfs}
\usepackage{calligra}
\usepackage{bbm}
\usepackage{tikz}
    \usetikzlibrary{arrows}
    \usetikzlibrary{positioning,arrows.meta}
    \usetikzlibrary{shapes,decorations}
    \usetikzlibrary{decorations.markings}
    \tikzset{
        vertLabel/.style={anchor=south, rotate=90, inner sep=.5mm},
        vertLabelSwap/.style={anchor=south, rotate=-90, inner sep=.5mm}
    }
    \tikzset{
        double line with arrow/.style args={#1,#2}{decorate,decoration={markings,%
        mark=at position 0 with {\coordinate (ta-base-1) at (0,1pt);
        \coordinate (ta-base-2) at (0,-1pt);},
        mark=at position 1 with {\draw[#1] (ta-base-1) -- (0,1pt);
        \draw[#2] (ta-base-2) -- (0,-1pt);
        }}}
    }
    \tikzset{
        Equal/.style={-,double line with arrow={-,-}}
    }
\usepackage{tikz-cd}
\usepackage{pgfplots}
    \pgfplotsset{compat=1.15}

\usepackage{rotating}
\usepackage{multirow}

\usepackage{comment}

\usepackage[doi = false, isbn = false, url = false, style = alphabetic, maxnames = 100]{biblatex}
    \DeclareFieldFormat{postnote}{#1}
    \AtBeginDocument{%
       \def\MR#1{}
    }

\usepackage{xr-hyper}
\usepackage[hidelinks]{hyperref}

\usepackage[capitalize]{cleveref}

\numberwithin{equation}{subsection}

\crefformat{equation}{(#2#1#3)}
\Crefformat{equation}{Equation~(#2#1#3)}
\crefmultiformat{equation}{(#2#1#3)}%
{ and~(#2#1#3)}{,~(#2#1#3)}{ and~(#2#1#3)}
\Crefmultiformat{equation}{Equations (#2#1#3)}%
{ and~(#2#1#3)}{,~(#2#1#3)}{ and~(#2#1#3)}
\crefrangeformat{equation}{(#3#1#4) to~(#5#2#6)}
\Crefrangeformat{equation}{Equations (#3#1#4) to~(#5#2#6)}

\crefformat{figure}{Figure~#2#1#3}
\Crefformat{figure}{Figure~#2#1#3}
\crefmultiformat{figure}{Figures (#2#1#3)}%
{ and~(#2#1#3)}{,~(#2#1#3)}{ and~(#2#1#3)}
\Crefmultiformat{figure}{Figures (#2#1#3)}%
{ and~(#2#1#3)}{,~(#2#1#3)}{ and~(#2#1#3)}
\crefrangeformat{figure}{Figures (#3#1#4) to~(#5#2#6)}
\Crefrangeformat{figure}{Figures (#3#1#4) to~(#5#2#6)}

\newtheorem{theorem}{Theorem}[subsection]
\newtheorem*{theorem*}{Theorem}

\newtheorem*{question*}{Question}
\newtheorem{lemma}[theorem]{Lemma}
\newtheorem{corollary}[theorem]{Corollary}
\newtheorem{proposition}[theorem]{Proposition}

\newtheoremstyle{defspaced}
  {.4\baselineskip\@plus.1\baselineskip\@minus.1\baselineskip} 
  {.4\baselineskip\@plus.1\baselineskip\@minus.1\baselineskip} 
  {\normalfont} 
  {}    
  {\bfseries} 
  {.}   
  { }   
  {}    

\theoremstyle{defspaced}

\newtheorem{remark}[theorem]{Remark}

\newtheorem*{notation*}{Notation}

\newtheorem{conj}[theorem]{Conjecture}

\makeatletter

\newcommand*{\addFileDependency}[1]{
\typeout{(#1)}
\@addtofilelist{#1}
\IfFileExists{#1}{}{\typeout{No file #1.}}
}

\makeatother

\newcommand{\citeOrEmpty}[2][]{%
  \ifnum\pdfstrcmp{#2}{}=0
    #1\ignorespaces
  \else
    \cite[{#1}]{#2}
  \fi
  \ignorespacesafterend
}

\ExplSyntaxOn
\NewDocumentCommand{\firstitem}{m}
 {
  \clist_item:nn { #1 } { 1 }
 }
\ExplSyntaxOff

\DeclareFontFamily{U}{stix2-mathex}{}
\DeclareFontShape{U}{stix2-mathex}{m}{n}{
  <-> stix2-mathex
}{}

\DeclareSymbolFont{stix2-mathexSymb}{U}{stix2-mathex}{m}{n}

\DeclareMathDelimiter{\stixlbrbrak}{\mathopen}{stix2-mathexSymb}{20}{stix2-mathexSymb}{20}
\DeclareMathDelimiter{\stixrbrbrak}{\mathclose}{stix2-mathexSymb}{21}{stix2-mathexSymb}{21}

\DeclareMathOperator{\Hom}{Hom}

\DeclareMathOperator{\GL}{GL}
\DeclareMathOperator{\SL}{SL}
\DeclareMathOperator{\Sp}{Sp}

\DeclareMathOperator{\Ind}{Ind}

\newcommand{\A}{\mathbb{A}}

\newcommand{\Z}{\mathbb{Z}}
\newcommand{\Q}{\mathbb{Q}}
\newcommand{\R}{\mathbb{R}}
\newcommand{\C}{\mathbb{C}}

\newcommand{\F}{\mathbb{F}}

\newcommand{\opn}{\operatorname}

\newcommand{\flatflat}{\flat\kern-1.4pt\flat}

\newcommand{\fa}{\mathfrak{a}}

\newcommand{\fg}{\mathfrak{g}}

\newcommand{\fm}{\mathfrak{m}}
\newcommand{\fn}{\mathfrak{n}}

\newcommand{\fs}{\mathfrak{s}}

\newcommand{\bA}{\mathbb{A}}

\newcommand{\bF}{\mathbb{F}}
\newcommand{\bG}{\mathbb{G}}

\newcommand{\cA}{\mathcal{A}}
\newcommand{\cB}{\mathcal{B}}

\newcommand{\cE}{\mathcal{E}}
\newcommand{\cF}{\mathcal{F}}
\newcommand{\cG}{\mathcal{G}}

\newcommand{\cL}{\mathcal{L}}

\newcommand{\cO}{\mathcal{O}}
\newcommand{\cP}{\mathcal{P}}
\newcommand{\cS}{\mathcal{S}}
\newcommand{\cT}{\mathcal{T}}
\newcommand{\cU}{\mathcal{U}}
\newcommand{\cV}{\mathcal{V}}
\newcommand{\cW}{\mathcal{W}}

\newcommand{\rd}{\mathrm{d}}

\newcommand{\tS}{\mathtt{S}}

\newcommand{\Ad}{\ensuremath{\mathrm{Ad}}}

\DeclareMathOperator{\gl}{\mathfrak{gl}}

\DeclareMathOperator{\Gr}{Gr}

\DeclareMathOperator{\Sym}{Sym}

\newcommand{\RNum}[1]{\uppercase\expandafter{\romannumeral #1\relax}}

\newlist{enumerateThm}{enumerate}{1}
\setlist[enumerateThm]{
  label=(\arabic*),
  ref=\thetheorem(\arabic*),
  }

\makeatletter
\apptocmd{\@begintheorem}{\let\thmtype\@currenvir}{}{\ERROR}
\makeatother

\newcounter{keepeqno}
\newenvironment{num}
 {\setcounter{keepeqno}{\value{equation}}%
  \begin{list}{(\theequation)}{\usecounter{equation}}%
  \setcounter{equation}{\value{keepeqno}}}
 {\end{list}}

\newcommand{\Bun}{\mathrm{Bun}}
\newcommand{\Shv}{\mathrm{Shv}}
\newcommand{\cusp}{\mathrm{cusp}}
\newcommand{\Vect}{\mathrm{Vect}}

\newcommand{\ev}{\mathrm{ev}}
\newcommand{\AS}{\mathrm{AS}}
\newcommand{\pr}{\mathrm{pr}}
\newcommand{\id}{\mathrm{id}}

\newcommand{\inj}{\hookrightarrow}
\newcommand{\surj}{\twoheadrightarrow}

\newcommand{\Mir}{\mathrm{Mir}}
\newcommand{\Ga}{\mathbb{G}_{\mathrm{a}}}

\newcommand{\CT}{\mathrm{CT}}
\newcommand{\FT}{\mathrm{FT}}

\newcommand{\sub}{\subset}
\newcommand{\om}{\omega}

\newcommand\tilM{\widetilde{M}}
\newcommand\tilN{\widetilde{N}}

\newcommand\tilc{\widetilde{c}}

\newcommand\tilf{\widetilde{f}}

\newcommand\tilq{\widetilde{q}}

\newcommand{\uk}{\underline{k}}

\usepackage{adjustbox}

\newcommand{\Loc}{\mathrm{Loc}}
\newcommand{\IndCoh}{\mathrm{IndCoh}}
\newcommand{\Nilp}{\mathrm{Nilp}}

\title{On the relative Langlands duality for $\Sp_{2n} \backslash \GL_{2n+1}$ (with an appendix by Zeyu Wang)}

\author{Weixiao Lu}
\address{Massachusetts Institute of Technology, Department of Mathematics, 77 Massachusetts Avenue, Cambridge, MA 02139, USA}
\email{weixiaol@mit.edu}

\author{Guodong Xi}
\address{University of Minnesota, Department of Mathematics, 206 Church St. S.E., Minneapolis, MN 55455, USA}
\email{xi000023@umn.edu}

\begin{document}

\maketitle

\date{\today}

\begin{abstract}
    We verify the relative Langlands duality conjecture proposed by Ben-Zvi, Sakellaridis, Venkatesh \cite{BZSV} for the hyperspherical Hamiltonian variety $T^*(\Sp_{2n}\backslash \GL_{2n+1})$. We provide numerical (over number fields and function fields) and geometric (in the \'{e}tale setting) evidence that its dual Hamiltonian variety should be $T^*(\GL_n \times \GL_{n+1} \backslash \GL_{2n+1})$ as is predicted by \cite{BZSV}.
\end{abstract}

\tableofcontents

\section{Introduction}

In the seminal work \cite{BZSV}, Ben-Zvi, Sakellaridis and Venkatesh (BZSV for short) proposed the well-known relative Langlands duality, which we may also call the BZSV duality. Let $G$ be a connected reductive group. The BZSV duality concerns the duality between certain $G$-Hamiltonian variety $M$ and $\widehat{G}$-Hamiltonian variety $\widehat{M}$, in the sense that the ``period" attached to $M$ on the automorphic side (A-side) should match with the ``$L$-function" attached to $\widehat{M}$ on the spectral side (B-side). One of the key features of the BZSV duality is that the double dual $\widehat{\widehat{M}}$ is expected to coincide with $M$. Therefore, we may switch the A-side and B-side and expect that the period associated to $\widehat{M}$ should also match with the $L$-function associated to $M$.

In the case of the polarized hyperspherical varieties in the sense of \cite[\S 3]{BZSV}, BZSV proposed a conjectural description of the dual varieties \cite[\S 4]{BZSV}. The structure of a hyperspherical variety is given by a quadruple $\Delta = (G,H,\rho_H,\iota)$,  Here $G$ is a split reductive group; $H$ is a reductive subgroup of $G$; $\rho_H$ is a symplectic representation of $H$; and $\iota$ is a homomorphism from $\SL_2$ into $G$ whose image commutes with $H$. In \cite[\S 4]{BZSV}, they proposed a combinatorial method to compute the dual of a polarized hyperspherical Hamiltonian variety that is attached to a ``dual quadruple" $\widehat{\Delta}=(\widehat{G}, \widehat{H}^\prime, \rho_{\widehat{H}^\prime}, \hat{\iota}^\prime)$.

The goal of this article is to verify some cases of the BZSV conjecture for the hyperspherical $\GL_{2n+1}$-variety $T^*(\Sp_{2n} \backslash \GL_{2n+1})$, or the quadruple $(\GL_{2n+1},\Sp_{2n},0,1)$. We show that its dual is $T^*(\GL_n \times \GL_{n+1} \backslash \GL_{2n+1})$, or attached to the quadruple $(\GL_{2n+1},\GL_n \times \GL_{n+1},0,1)$ as predicted by \cite[\S 4]{BZSV}. The conjecture has various setting; we mainly concentrate on the global numerical and the global geometric setting. We now describe them in detail.

\subsection{Numerical Result}

We briefly recall the BZSV conjecture in the numerical (also known as the classical Langlands) setting (see also \cite{MWZ1}). The map $\iota$ induces an adjoint action of $H \times \SL_2$ on the Lie algebra $\mathfrak{g}$ of $G$ and we can decompose it as
\[
    \mathfrak{g}= \bigoplus_{ k \in I} \rho_k \otimes \Sym^k 
\]
where $\rho_k$ is some representation of $H$ and $I$ is a finite subset of $\mathbb{Z}_{\geq 0}$. Let $I_{\mathrm{odd}}$ denote the subset of $I$ containing all the odd numbers and let
\[
    \rho_{H, \iota}= \rho_H \oplus \left(\bigoplus_{k \in I_{\mathrm{odd}}} \rho_k \right).
\]

Let $F$ denote a global field and let $\mathbb{A}=\mathbb{A}_F$. For a BZSV quadruple, $\rho_{H, \iota}$ is a symplectic anomaly-free representation of $H$. Take a maximal isotropic subspace $Y$ of $\rho_{H, \iota}$. For a Schwartz function $\Phi$ on $Y(\mathbb{A})$, the previous condition ensures we can define the associated theta series $\Theta(h, \Phi)$ on $H(\mathbb{A})$.

For automorphic form $\varphi$ on $G(\mathbb{A})$, let $\cP_\iota(\varphi)$ denote the degenerate Fourier coefficient associated to $\iota$. We define the following period integral
\[
    \cP_{H, \iota, \rho_H}(\varphi, \Phi)=\int_{[H]}\cP_\iota(\varphi)(h)\Theta(h, \Phi)\rd h
\]
whenever the integral is convergent.
The following conjecture (\cite[\S 14]{BZSV}, \cite[\S 1.1]{MWZ1}) is the main conjecture regarding this period integral.
\begin{conj}
    Let $\pi$ be an irreducible automorphic representation of $G(\mathbb{A})$. Then the period integral $\cP_{H, \iota, \rho_H}(\varphi)$ is nonzero only if the Arthur parameter of $\pi$ factors through
    \[
        \hat{\iota}^\prime : \widehat{H}^\prime(\mathbb{C}) \times \SL_2(\mathbb{C}) \to \widehat{G}(\mathbb{C}).
    \]
    If this is the case and $\pi$ is lifting of a global tempered Arthur packet of $H^\prime(\mathbb{A})$, then we have
    \begin{equation} \label{eq:MWZ}
         \frac{|\cP_{H, \iota, \rho_H}(\varphi)|^2}{\langle \varphi, \varphi \rangle} ``=" \frac{\displaystyle L \left(\frac{1}{2}, \Pi, \rho_{\widehat{H}^\prime}\right)\cdot \prod_{k \in \hat{I}}L \left( \frac{k}{2}+1, \Pi, \hat{\rho}_k \right)}{L(1, \Pi, \Ad)^2}.
    \end{equation}
    Here $\langle \, , \, \rangle$ denotes a certain version of $L^2$-norm.
\end{conj}

\begin{remark}
    Note that in \cite{MWZ1}, the authors state the conjecture for only discrete $\pi$, in which case $\langle \, , \, \rangle$ denotes the standard $L^2$-norm. However, it's believed that the statement should hold for general $\pi$, as stated in the above conjecture.
\end{remark}

If we consider the BZSV quadruple $(\GL_{2n+1}, \GL_n \times \GL_{n+1}, 0, 1)$. In \cite{LinearPeriods}, the authors verified that the associated linear period vanishes for cuspidal automorphic representations of $\GL_{2n+1}(\mathbb{A})$. Moreover, let $\pi$ be a cuspidal automorphic representation of $\GL_{2n}(\mathbb{A})$ whose 
hypothetical Langlands parameter factors through $\Sp_{2n}(\mathbb{C})$ and let $1$ denote the trivial representation of $\GL_1$. Then in \cite[Theorem 5.1]{LinearPeriods}, the authors verified that if the automorphic representation of $\GL_{2n+1}(\mathbb{A})$ is $\pi \boxplus 1$, the regularized period integral defined in \emph{loc. cit.} represents the standard $L$-function $L(1, \pi)$. This coincides with the claimed identity in the previous conjecture. The results above suggest the BZSV dual of $(\GL_{2n+1}, \GL_n \times \GL_{n+1}, 0, 1)$ should be $(\GL_{2n+1}, \Sp_{2n}, 0, 1)$, as is predicted by \cite{BZSV}.

In this article, we switch the A-side and the B-side of the linear period mentioned above. Consider the BZSV quadruple $(\GL_{2n+1}, \Sp_{2n}, 0, 1)$, we verify that its dual quadruple should be $(\GL_{2n+1},\GL_n \times \GL_{n+1},0,1)$. More concretely, we prove the following result:

\begin{theorem} \label{thm:intro_main}
    Let $\pi$ be an automorphic representation of $\GL_{2n+1}(\bA_F)$, then
    \begin{enumerate}
        \item If $\pi$ is cuspidal, or $\pi$ is of the form $\Pi_1 \boxplus \Pi_2$, where $\Pi_i$ are cuspidal automorphic representation of $\GL_{n_i}(\bA_F)$ with central character trivial on $A_{\GL_{n_i}}^\infty$ and $n_1+n_2 = 2n+1$ with $1 \le n_i \le n-1$. Then for any $f \in \pi$, we have
        \begin{equation*}
            \int_{[\Sp_{2n}]} f(h) \rd h = 0. 
        \end{equation*}
        \item If $\pi$ is of the form $\Pi_n \boxplus \Pi_{n+1}$, where $\Pi_i$ is a cuspidal automorphic representation of $\GL_i(\bA_F)$ with central character trivial on $A_{\GL_i}^\infty$ for $i=n,n+1$. Let $\Pi := \Pi_n \boxtimes \Pi_{n+1}$ be the cuspidal automorphic representation of $\GL_n(\bA_F) \times \GL_{n+1}(\bA_F)$, then for $f = \otimes f_v \in \pi$, we have
        \begin{equation} \label{eq:main}
           \frac{\lvert \cP(f) \rvert^2}{\langle f, f \rangle_{\mathrm{Pet}}}= \frac{\zeta^*(1)\zeta(3)\cdots \zeta(2n+1)}{\zeta(2) \cdots \zeta(2n)} \frac{L^*(1, \Pi, \hat{\rho}_0)}{L^*(1, \Pi, \Ad)^2}\prod_{v}\frac{|\cP_{v}^\natural(f_v)|^2}{\langle f_v , f_v \rangle^\natural}.
        \end{equation}
    \end{enumerate}
\end{theorem}
The reader should note the formal resemblance of \eqref{eq:main} with \eqref{eq:MWZ}. We explain the terms in \eqref{eq:main} in some detail.

\begin{itemize}
    \item Let $P$ be the standard parabolic subgroup of $\GL_{2n+1}$ with Levi $\GL_n \times \GL_{n+1}$. Then by $f \in \pi$ we mean $f$ is an Eisenstein series of the form $E(g,\varphi) := E(g,\varphi,0)$, where $\varphi \in \Ind_{P(\bA)}^{\GL_{2n+1}(\bA)} \Pi$ (see \S \ref{ssec:automorphic_forms} for more details). We will show in \S \ref{sec:convergence} that 
    \begin{equation*}
       \cP(E(\cdot,\varphi)) := \int_{[\Sp_{2n}]} E(h,\varphi) \rd h
    \end{equation*}
    is absolutely convergent. And we put
    \begin{equation*}
        \langle f,f \rangle_{\opn{Pet}} := \langle \varphi,\varphi \rangle_{\opn{Pet}} := \int_{[\GL_{2n+1}]_{P,0}} \lvert \varphi(x) \rvert^2 \rd x .
    \end{equation*}
    This explains the left hand side of \eqref{eq:main}.
    \item Let $\zeta(s) = \zeta_F(s)$ denote the completed Dedekind zeta function of $F$ and $\zeta^*(1)$ denotes the residue of $\zeta(s)$ at $s=1$.
    \item $L^*(1,\Pi,\opn{Ad})$ denotes the leading term of the Laurant expansion of $L^*(s,\Pi,\opn{Ad})$ at $s=1$. Similar for $L^*(1,\Pi,\hat{\rho}_0)$. Note that in this case, we have
    \begin{equation*}
        L^*(1,\Pi,\hat{\rho}_0) = L^*(1,\Pi_n,\opn{Ad}) L^*(1,\Pi_{n+1},\opn{Ad}) L(1,\Pi_n^{\vee} \times \Pi_{n+1}) L(1,\Pi_n \times \Pi_{n+1}^{\vee}),
    \end{equation*}
    and
    \begin{equation*}
        L^*(1,\Pi,\opn{Ad}) = L^*(1,\Pi_n,\opn{Ad}) L^*(1,\Pi_{n+1},\opn{Ad}).
    \end{equation*}
    \item $\cP^{\natural}_v(f_v)$ and $\langle f_v,f_v \rangle^{\natural}$ are local (normalized) version of the period and inner product respectively. They are defined in \S \ref{ssec:proof_of_intro_main} in terms of Whittaker model, with the properties that when everything is unramified, then $\cP^{\natural}_v(f_v) = \langle f_v, f_v \rangle^{\natural} = 1$.
\end{itemize}

Finally, we remark that the case $\Sp_{2n} \backslash \GL_{2n}$, commonly referred to as the \emph{symplectic period}, has been extensively studied by Jacquet–Rallis \cite{JR92} and Offen \cite{Offen06}. Its (expected) BZSV dual is the so-called \emph{Jacquet–Shalika period}.

\subsection{Geometric result}

A remarkable feature of conjectures in \cite{BZSV} is their parallelity over number fields and function fields, enabling predictions of results over number fields from parallel but sometimes easier results over function fields. Over function fields, beyond the numerical conjecture mentioned before, one can ask a matching between geometrizations of both sides of the identities in theorem \ref{thm:intro_main} under the conjectural geometric Langlands equivalence $\Shv_{\Nilp}(\Bun_G)\cong\IndCoh_{\Nilp}(\Loc_{\widehat{G}})$. We refer to \cite{arinkin2022stacklocalsystemsrestricted} for a precise formulation of this equivalence. 

A precise conjecture in this direction is formulated in \cite[\S 12]{BZSV}, which we briefly recall in the case that $M=T^*X$ and $X=G/H$ where $H\sub G$ is a connected reductive subgroup making $M$ hyperspherical (i.e. we are considering quadruple $(G,H,0,1)$): On the A-side, one considers map $\pi:\Bun_H\to\Bun_G$ and defines the \emph{period sheaf} \[\cP_X:=\pi_!\overline{\mathbb{Q}}_l\] (we work with constructible \'etale $\overline{\mathbb{Q}}_l$-complexes), which induces a natural functor \[\ev^{\Mir}(\cP_X\otimes-)\cong \Gamma_c(\pi^*(-)):\Shv_{\Nilp}(\Bun_G)\to\Vect\] where $\ev^{\Mir}:\Shv(\Bun_G)^{\otimes 2}\to \Vect$ is the evaluation map for the miraculous duality of $\Shv(\Bun_G)$ defined in \cite{AGKRRV2}. On the B-side, one can similarly define the \emph{$L$-sheaf} \[\cL_{\widehat{M}}\in\IndCoh(\Loc_{\widehat{G}})\] using the dual $\widehat{G}$-Hamiltonian space $\widehat{M}$ and similarly consider the functor \[\ev^{\mathrm{Serre}}(\cL_{\widehat{M}}\otimes -):\IndCoh_{\Nilp}(\Loc_{\widehat{G}})\to \Vect\] where $\ev^{\mathrm{Serre}}:\IndCoh(\Loc_{\widehat{G}})^{\otimes 2}\to\Vect$ is the evaluation map for the Grothendieck-Serre duality on $\Loc_{\widehat{G}}$. The conjecture in \cite[\S 12]{BZSV} predicts an isomorphism between functors $\Shv_{\Nilp}(\Bun_G)\cong\IndCoh_{\Nilp}(\Loc_{\widehat{G}}) \to \Vect$: \[\ev^{\Mir}(\cP_X\otimes-)\cong \ev^{\mathrm{Serre}}(\cL_{\widehat{M}}\otimes c^*(-))\] where $c:\Loc_{\widehat{G}}\to\Loc_{\widehat{G}}$ is induced by the Cartan involution on $\widehat{G}$.

In the appendix, we verify the conjecture above for pairs of groups $(G,H)=(\GL_{2n+1},\Sp_{2n})$ and $(G,H)=(\GL_{2n},\Sp_{2n})$ on the cuspidal part. More precisely, we prove the following:

\begin{theorem} \label{thm:intro_geo}
    For $(G,H)$ as above, for any cuspidal automorphic sheaf $\cF\in\Shv(\Bun_G)_{\cusp}$, one has \[\Gamma_c(\pi^*\cF)=0.\]
\end{theorem}

\begin{remark}
    Such type of geometric results in many other examples are considered in series works of Lysenko (an incomplete list includes \cite{lysenko2002local}\cite{lysenko2008geometric}\cite{Lysenko_2021}) and has their potential applications to arithmetic period conjectures as discusses in \cite{LW}.
\end{remark}

\subsection{Structure of the article}

After introducing notations and preliminaries in \S 2. We study the the numerical conjecture in the main part of the article and geometric conjecture in the appendix.

In \S \ref{sec:convergence}, we show that the periods of the Eisenstein series we considered are absolutely convergent. We treat the function field case and the number field case separately. In \S \ref{sec:cuspidal} we prove that the $\Sp_2n$ periods of cuspidal automorphic forms, as well as certain kinds of cuspidal Eisenstein series, on $\GL_{2n+1}$ vanish. In \S \ref{sec:orbits}, we do some linear algebra to classify $\Sp_{2n}$ orbits on $\GL_{2n+1}/P_{n,n+1}$, where $P_{n,n+1}$ is the maximal parabolic subgroup of $\GL_{2n+1}$ with type $(n,n+1)$. In \S \ref{sec:proof}, we use the results in \S \ref{sec:cuspidal} and \S \ref{sec:orbits} to prove our main numerical results.

In the appendix, we geometrize the results in \S \ref{sec:cuspidal} to get the geometric conjecture.

\subsection{Acknowledgement}

We are grateful to Chen Wan for the kind suggestion of this article and for many helpful discussions. We thank Zeyu Wang for generously contributing the appendix to our work and for carefully proofreading the manuscript. We thank Paul Boisseau, Dihua Jiang, Yannis Sakellaridis, Yiyang Wang, Hang Xue, Wei Zhang for helpful discussions and suggestions. 

\section{Preliminaries and Notations}
\label{sec:prelim}

\subsection{General notations}

We list some general notations:

\begin{itemize}
     \item For a matrix $A$, we write ${}^tA$ for the transpose of $A$. 

    \item Let $F$ be a global field and let $v$ be a place of $F$, we write $F_v$ for the completion of $F$ at the place $v$. In general, if $\tS$ is a finite set of places of $F$, we write $F_\tS := \prod_{v \in \tS} F_v$ and $\bA^\tS_F$ for the restricted product $\prod_{v \not \in \tS}' F_v$. We also write $F_\infty := F \otimes_\Q \R$. 

    \item Let $f,g$ be two positive functions on a set $X$, we write $f \ll g$ if there exists $C>0$ such that $f(x) \le Cg(x)$ for any $x \in X$. 
\end{itemize}

\subsection{Algebraic groups and their ad\`{e}lic points}

In this subsection, we let $G$ be a connected linear algebraic group over a global field $F$. We denote by $\bA := \bA_F$ and $[G] := G(F) \backslash G(\bA)$ the ad\`{e}lic quotient of $G$.

\subsubsection{Tamagwa measure} \label{ssec:Tamagawa}

We fix the Tamagawa measure $dg$ on $G(\bA)$, and thus on $[G]$ as described in \cite[section 2.3]{BPCZ}. To fix the notations, we briefly recall the definition. Fix an additive character $\psi:F \backslash \bA_F \to \C^\times$. Write $\psi$ as $\psi = \prod \psi_v$. For each place $v$ of $F$, $\psi_v$ determines the self-dual measure on $F_v$. Let $\omega$ be an $F$-rational $G$-invariant top differential form on $G$. For each place $v$, $\lvert \omega \rvert$ gives a measure $d^* g_v$ on $G(F_v)$. Moreover, according to the results of Gross \cite{Gross97}, there exists a global Artin-Tate $L$-function $L_G(s)$ such that 
\begin{equation*}
    d^*g_v(G(\cO_v)) = L_{G,v}(0)
\end{equation*}
for almost all places $v$. We denote by
\begin{equation*}
    \Delta_{G,v} := L_{G,v}(0)
\end{equation*}
and let $\Delta_G^*$ denote the leading coefficient of the Laurent expansion of $L_G(s)$ at $s=0$.
We then put $dg_v := \Delta_{G,v} d^* g_v$ and put the global Tamagawa measure by
\begin{equation*}
    dg = (\Delta_G^*)^{-1} \prod_v d g_v.
\end{equation*}
The measure is independent of the choice of $\omega$.

When $G=\GL_n$, we can take $\omega = (\wedge dg_{ij})/(\det g)^n$ and we have $L_G(s) = \zeta_F(s+1) \cdots \zeta_F(s+n)$, where $\zeta_F$ denote the (completed) Dedekind zeta function of $F$. 

When $G=\Sp_{2n}$, we have $L_G(s) = \zeta_F(s+2) \cdots \zeta_F(s+2n)$.

\subsubsection{Parabolic subgroups}
 We assume that $G$ is connected and reductive for the remainder of this subsection. Fix a maximal split torus $A_0$ of $G$ and a minimal parabolic subgroup $P_0$ containing $A_0$. A parabolic subgroup $P$ of $G$ is called \emph{standard} if $P \supset P_0$, and \emph{semi-standard} if $P \supset A_0$. For a semi-standard parabolic subgroup $P$, we denote by $P=M_PN_P$ the standard Levi decomposition of $P$. 

We denote by $\cF := \cF^G$ the set of semi-standard parabolic subgroups of $G$. For a subgroup $S$ of $G$, we denote by $\cF(S) = \cF^G(S)$ the set of semi-standard parabolic subgroups containing $S$. For example, $\cF(A_0)$ (resp. $\cF(P_0)$) is the set of semi-standard (resp. standard) parabolic subgroups.

Recall that for any cocharacter $\lambda:\bG_m \to A_0$, the dynamical method associates a semi-standard parabolic subgroup $P(\lambda)$ of $G$:
    \begin{equation} \label{eq:dynamical}
        P(\lambda) := \{ g \in G \mid \lim_{t \to 0} \lambda(t)g \lambda(t)^{-1} \text{ exists} \},
    \end{equation}
    and all semi-standard parabolic subgroups are of the form $P(\lambda)$ for some $\lambda \in X_*(A_0)$.

    Let $W$ be the Weyl group of $(G, A_0)$, that is, the quotient by $M_0(F)$ of the normalizer of $A_0$ in $G(F)$. For standard parabolic subgroups $P,Q$, denote by
    \begin{equation*}
        {}_Q W {}_P := \{ w \in W \mid M_P \cap w^{-1}P_0w = M_P \cap P_0, \quad M_Q \cap wP_0w^{-1} = M_Q \cap P_0 \}.
    \end{equation*}
    It forms a representative of the double coset $W^Q \backslash W/W^P$. For $w \in {}_Q W_P$, $M_P \cap w^{-1}M_Qw$ is the Levi factor of the standard parabolic subgroup $P_w = (M_P \cap w^{-1}Qw)N_P$. In the same way, $M_Q \cap wM_Pw^{-1}$ is the Levi factor of the standard parabolic subgroup $Q_w = (L \cap wPw^{-1})N_Q$. We have $P_w \subset P$, $Q_w \subset Q$, moreover $P_w$ and $Q_w$ are associate. We then write
    \begin{equation*}
        W(P;Q) = \{ w \in {}_Q W_P \mid M_P \subset w^{-1}M_Qw \}.
    \end{equation*}
    
    For a semi-standard parabolic subgroup $P$ of $G$, define
        \[
            \fa_P^* := X^*(P) \otimes_\Z \R ,\quad \fa_P := \Hom_\Z(X^*(P),\R).
        \]
    We endow $\fa_P$ with the Haar measure such that the lattice $\Hom(X^*(P),\Z)$ has covolume 1.
        
    Let $\fa_0 := \fa_{P_0}$ and $\fa_0^* := \fa_{P_0}^*$. 
        \[
            \epsilon_P := (-1)^{\dim \fa_P - \dim \fa_G}.
        \]
    
    For any semi-standard parabolic $P$, let $\widehat{\tau}_P$ be the characteristic function of a cone on $\fa_P$ defined in \cite[\S 5]{Arthur78}

    \subsubsection{Norms}

    Now we assume that $F$ is a number field. We write $\fg_\infty$ for the Lie algebra of the Lie group $G(F_\infty)$ and write $\cU(\fg_\infty)$ for the universal enveloping algebra of $\fg_\infty$.
    
     For a semi-standard parabolic subgroup $P$ of $G$, we put
        \[
            [G]_P := N_P(\bA)M_P(F) \backslash G(\bA).
        \]
    We fix a norm $\| \cdot \|$ on $G(\bA)$ as in ~\cite[Appendix A]{BP21}. It induces a norm on $[G]_P$ by 
        \[
            \| g \|_P := \inf_{\gamma \in N_P(\bA)M_P(F)} \|\gamma g\|.
        \]
    There is a notion of weight functions on $[G]_P$ described in ~\cite[\S 2.4.3]{BPCZ}. In particular, for any $\alpha \in \fa_0^*$, there is a weight $d_{P,\alpha}$ on $[G]_P$.

    \subsubsection{\texorpdfstring{The map $H_P$}{The map H_P}}

     We denote by $A_G^\infty$ the neutral component of real points of the maximal split central torus of $\mathrm{Res}_{F/\Q} G$. For a semi-standard parabolic subgroup $P$ of $G$, let $A_P^\infty := A_{M_P}^\infty$. We also define $A_0^\infty := A_{P_0}^\infty = A_{M_0}^\infty$.

      Let $\delta_P: P(\bA) \to \R_{>0}$ denote the modular function of $P(\bA)$.

    We fix a maximal compact subgroup $K$ of $G(\bA)$, which is in good position with $P_0$. Hence, we have the Iwasawa decomposition $G(\bA)=P(\bA)K$ for all semi-standard parabolic subgroups $P$ of $G$. The map
        \[
            H_P: P(\bA) \to \fa_P, \, p \mapsto \left( \chi \mapsto \log \lvert \chi(g) \rvert  \right), \quad \chi \in X^*(P),
        \]
    extends to $G(\bA)$, by requiring it trivial on $K$. The map $H_P$ induces an isomorphism $A_P^\infty \cong \fa_P$, we endow $A_P^\infty$ with the Haar measure such that this isomorphism is measure-preserving. 

    Let $[G]_{P,0} := N(\bA)M(F) A_P^{\infty} \backslash G(\bA)$. For a measurable function $\varphi: [G]_P \to \C$ such that $\varphi(ag)=\delta_P(a)^{\frac 12} \varphi(g)$, the integral
    \begin{equation*}
        \int_{[G]_{P,0}} \lvert \varphi(x) \rvert \rd x
    \end{equation*}
    makes sense, and when it is finite, the integral
    \begin{equation*}
        \int_{[G]_{P,0}} \varphi(x) \rd x 
    \end{equation*}
    makes sense.

     \subsubsection{Siegel sets}

    By a Siegel set $\fs^P$ of $[G]_P$ we mean a subset of
$[G]_P$ of the form
    \[
    \fs^P = \omega_0 \{ a \in A_0^\infty \mid
    \langle \alpha, H_0(a) - T_- \rangle \geq 0, \ \alpha \in \Delta_0^P \} K,
    \]
where $T_- \in \fa_0$, and such that $G(\bA) = P(F)\fs^P$. We
assume that, for different parabolic subgroups of $G$, their Siegel sets
are defined by the same $T_-$. In particular if $P \subset Q$ then $\fs^P
\supset \fs^Q$.

     \subsubsection{Truncation parameters}
          
     Fix a norm $\| \cdot \|$ on $\fa_0$. We say $T \in \fa_0$ is \emph{sufficiently positive}, if
    there exists $C>0$ and $\varepsilon >0$ such that
    \[
            \inf_{\alpha \in \Delta_0} \langle \alpha,T \rangle \ge \max\{ C,\varepsilon \|T\| \}.
        \]

    \subsubsection{Paley-Wiener spaces}

    For a semi-standard parabolic $P$, let $\opn{PW}(i\fa_P^*)$ denote the space of Paley-Wiener functions on $i\fa_P^*$. More precisely, it consists of entire function $f$ on $\fa_{P,\C}^*$ with the following condition:
    \begin{num}
       \item \label{eq:Paley-Wiener} There exists $A>0$ such that for all integer $N>0$,
        \begin{equation*}
            \lvert f(\lambda) \rvert \ll (1+\lvert  \lambda \rvert)^{-N} e^{A \lvert \opn{Re}(\lambda) \rvert}
        \end{equation*}
    \end{num}
    
    The Fourier transform:
    \begin{equation*}
        \opn{PW}(i\fa_P^*) \ni f \mapsto \widehat{f}(X) = \int_{i\fa_P^*} f(\lambda) e^{\langle \lambda,X \rangle} \rd \lambda
    \end{equation*}
    defines a bijection between $\opn{PW}(i\fa_P^*)$ and $C_c^\infty(\fa_P)$.

\subsection{Function spaces} \label{subsec:function space}
    Let $F$ be a number field and let $G$ be an algebraic group over $F$. Let $P$ be a parabolic subgroup of $G$. There are various function spaces on $[G]_P$ which we briefly recall below. The reader may consult ~\cite[\S 2.5]{BPCZ} for more details.

    A function $f:G(\bA) \to \C$ is called \emph{smooth}, if it is right $J$-invariant for some open compact subgroup $J \subset G(\bA_f)$ and for any $g_f \in G(\bA_f)$, the function $g_\infty \mapsto f(g_f g_\infty)$ is $C^\infty$. A function on $[G]_P$ is called smooth if it pullbacks to a smooth function on $G(\bA)$.

    Let $\cS([G]_P)$ be the space of \emph{Schwartz functions} on $[G]_P$. It is the union of $\cS([G]_P,J)$ for open compact subgroup $J \subset G(\bA_f)$. Where $\cS([G]_P,J)$ is the space of smooth functions on $[G]_P$ which are right $J$ invariant and 
    \[
        \| f \|_{X,N} := \sup_{x \in [G]_P} \lvert R(X)f(x) 
        \rvert \| x \|_P^N < \infty
    \]
    for any $X \in \cU(\fg_{\infty})$ and $N>0$. The vector space $\cS([G]_P,J)$ is naturally a Fr\'{e}chet space and $\cS([G]_P)$ is naturally a strict LF space.

    Let $\cS^0([G]_P)$ be the space of measurable function $f$ on $[G]_P$ such that
    \begin{equation} \label{eq:norm_infty_N}
       \| f\|_{\infty,N} := \sup_{x \in [G]_P} \lvert f(x) \rvert \| x \|^{N}_P < \infty
    \end{equation}
    for any $N>0$. It is naturally a Fr\'{e}chet space.

Let $\cT([G]_P)$ be the function of \emph{uniform moderate growth} on $[G]_P$. It is the union of $\cT_{-N}([G]_P,J)$, where $N>0$ and $J \subset G(\bA_f)$ is open compact subgroup. $\cT_{-N}([G]_P,J)$ consists of smooth functions $f$ on $[G]_P$ which are right $J$-invariant and
\begin{equation*}
    \| f \|_{X,-N} := \sup_{x \in [G]_P} \lvert R(X)f(x) \rvert \| x \|_P^{-N} < \infty
\end{equation*}
for any $X \in \cU(\fg_\infty)$. The vector space $\cT_{-N}([G]_P,J)$ is naturally a Fr\'{e}chet space and $\cT([G]_P)$ then carries the induces (non-strict) LF topology.

For $P \subset Q$, we have the following \emph{constant term} map
\begin{equation*}
    \cT([G]_Q) \ni \varphi \mapsto \varphi_P := \left( g \mapsto \int_{[N_P]} \varphi(ng) \rd n \right) \in \cT([G]_P).
\end{equation*}

Let $\cT_{\opn{cusp}}([G]_Q)$ be the closed subspace of $\cT([G]_Q)$ consists of $\varphi$ such that $\varphi_P = 0$ for any $P \subset Q$. It is well known that for $\varphi \in \cT_{\opn{cusp}}([G])$, we have the following \emph{rapid decay property of cusp form}:
\begin{num} 
    \item \label{eq:cusp_form_rapid_decay} For any $N>0$ and any $X \in \cU(\fg_{\infty})$, we have $\sup_{x \in [G]^1} \lvert R(X)\varphi(x) \rvert \| x \|_G^N < \infty$. 
\end{num}

For standard parabolic subgroups $P \subset Q$ , the pseudo-Eisenstein series $E := E_P^Q$ defines a map
\begin{equation*}
    E_P^G: \cS([G]_P) \to \cS([G]_Q), \quad \varphi \mapsto \sum_{\gamma \in P(F) \backslash Q(F)} \varphi(\gamma g).
\end{equation*}
If $P$, $Q$ be two parabolic subgroups of $G$, for $\varphi \in \cS_{\opn{cusp}}([G]_P)$, we have the formula ~\cite[\S II.1.7]{MW95}
\begin{equation} \label{eq:constant_term_pseudo}
    E_Q(g,\varphi) = \sum_{w \in W(P;Q)} E^Q(g,M(w)\varphi),
\end{equation}
where $M(w)$ denotes the intertwining operator:
\begin{equation*}
    M(w): \cS([G]_{P_w}) \to \cS([G]_{Q_w}), M(w)\varphi(g) = \int_{N_{P_w}(\bA) \cap wN_Pw^{-1}(\bA) \backslash N_{P_w}(\bA)} \varphi(w^{-1}ng) \rd n.
\end{equation*}

\subsection{Automorphic forms and Eisenstein series}
\label{ssec:automorphic_forms}

\subsubsection{Automorphic forms}

For a semi-standard parabolic subgroup $P$ of $G$, we denote by $\cA_P(G)$ the \emph{automorphic forms} on $[G]_P$, it consists of $Z(\fg_\infty)$ elements in $\cT([G]_P)$ in the number field case, and $\cA_P(G) = \cT([G]_P)$ in the function field case. When $P=G$, we denote by $\cA(G) := \cA_G(G)$.

Let $\cA_{P,\mathrm{cusp}}(G) := \cA_P(G) \cap \cT_{\mathrm{cusp}}([G]_P)$ be the subspace of \emph{cusp forms}. Then spaces $\cA_P(G)$ and $\cA_{P,\mathrm{cusp}}(G)$ carry natural LF topologies as described in \cite[\S 2.7.1]{BPCZ}.

By a \emph{cuspidal automorphic representation} $\pi$ of $M_P(\bA)$, we mean a topologically irreducible subrepresentation of $\cA_{\mathrm{cusp}}(M_P)$. For $\lambda \in \fa_{P,\C}^*$, let $\pi_\lambda$ be the space of functions
\begin{equation*}
    [M_P] \ni m \mapsto e^{\langle \lambda,H_P(m) \rangle} \varphi(m), \quad \varphi \in \pi.
\end{equation*}

Let $\pi$ be a cuspidal automorphic representation of $M_P(\bA)$ such that the central character is trivial on $A_{P}^\infty$. We denote by $\opn{Ind}_{P(\bA)}^{G(\bA)}$ the \emph{normalized smooth induction} of $\pi$, realizes as a subspace of $\cA_P(G)$ consisting of elements $\varphi \in \cA_P(G)$ such that for any $g \in G(\bA)$,
\begin{equation*}
    [M_P] \ni m \mapsto e^{\langle -\rho_P, H_P(m) \rangle} \varphi(mg)
\end{equation*}
 belongs to $\pi$. Similarly, let $\cA_{P,\pi}$ be the subspace of $\varphi_P(G)$ such that for any $g \in G(\bA)$,
\begin{equation*}
    [M_P] \ni m \mapsto e^{\langle -\rho_P, H_P(m) \rangle} \varphi(mg)
\end{equation*}
 belongs to the $\pi$-isotypic part of $\cA_{\mathrm{cusp}}(M_P)$.

 Let $G=\GL_n$ and $P=P_{n_1,n_2,\cdots,n_k}$ be the standard parabolic with Levi $\prod_{i=1}^{n_i} \GL_{n_i}$. If a cuspidal automorphic representation $\pi$ of $M_P$ (with central character trivial on $A_P^\infty$)  is of the form $\pi_1 \boxtimes \pi_2 \cdots \boxtimes \pi_k$, where $\pi_i$ is a cuspidal automorphic representation of $\GL_{n_i}(\bA)$, we will also write
    \begin{equation*}
        \pi_1 \boxplus \cdots \boxplus \pi_k := \opn{Ind}_{P(\bA)}^{\GL_n(\bA)} \pi = \cA_{P,\pi}.
    \end{equation*}

\subsubsection{Eisenstein series}

Let $P$ be a semi-standard parabolic subgroup of $G$. For any $\varphi \in \cA_P(G)$, $g \in G(\bA)$ and $\lambda \in \fa_{P,\C}^*$, the \emph{Eisenstein series} is defined by
\begin{equation*}
    E(g,\varphi,\lambda) := \sum_{\gamma \in P(F) \backslash G(F)} \varphi_{\lambda}(\gamma g),
\end{equation*}
where
\begin{equation*}
    \varphi_{\lambda}(g) = \varphi(g) e^{\langle \lambda, H_P(g) \rangle }.
\end{equation*}

By the result of \cite{Langlands76}, \cite{BL24}, the sum defining $E(g,\varphi,\lambda)$ is convergent when $\opn{Re}(\lambda)$ is sufficiently positive, and admits meromorphic continuation to $\fa_{P,\C}^*$.

Let $\pi$ be a cuspidal automorphic representation of $M_P(\bA)$ with trivial central character on $A_P^\infty$. We say $\lambda$ is regular for $\pi$, if $E(g,\varphi,\lambda)$ is holomorphic at $\lambda$ for every $\varphi \in \cA_{P,\pi}$ and $g \in [G]$. It is known that for any such $\pi$, any $\lambda \in i \fa_P^*$ is regular for $\pi$.

Let $P,Q$ be two standard parabolic subgroups of $G$. For $w \in W(P,Q)$ and $\lambda \in \fa_{P,\C}^*$. The \emph{intertwining operator} is defined by
\begin{equation*}
    M(w,\lambda): \cA_P(F) \ni \varphi \mapsto M(w,\varphi)(g) := e^{\langle -w\lambda, H_Q(g) \rangle} \int_{N_Q \cap wN_Pw^{-1}(\bA) \backslash N_Q(\bA)} e^{\langle \lambda,H_P(w^{-1}\lambda g) \rangle} \varphi(w^{-1}ng).
\end{equation*}

By the result of \cite{Langlands76}, \cite{BL24}, the intertwining operator is convergent when $\opn{Re}(\lambda)$ is sufficiently positive and has meromorphic continuation to $\fa_{P,\C}^*$.

For $\varphi \in \cA_{P,\mathrm{cusp}}(G)$, we have the constant term formula analogous to \eqref{eq:constant_term_pseudo}:
\begin{equation} \label{eq:constant_term}
    E_Q(g,\varphi,\lambda) = \sum_{w \in W(P;Q)} E^Q(g,M(w)\varphi,w\lambda).
\end{equation}

\subsection{\texorpdfstring{$\Sp_{2n} \subset \GL_{2n+1}$}{Sp_{2n} in GL_{2n+1}}}

Throughout the remaining part of the article, we fix an integer $n$ and let $G=\GL_{2n+1}$ be the general linear group acting on $F^{2n+1}$. Identify $F^{2n}$ with the subspace of $F^{2n+1}$ with the last coordinate equal to $0$. Then we can identify $\GL_{2n}$ with a subgroup of $G$ at the top-left corner. Let $J := J_n$ denote the following matrix
\[
    J = \begin{pmatrix}
 &  &  &  & 1 \\
 &  &  & 1 &  \\
 &  & \cdots &  &  \\
 & 1 &  &  &  \\
1 &  &  &  & 
\end{pmatrix}.
\]

We define an involution $\theta$ on $\GL_{2n}$ by
\begin{equation}\label{involution}
    \theta(g)=\begin{pmatrix}
 & -J \\
J & 
\end{pmatrix} {}^{t}g^{-1}\begin{pmatrix}
 & J \\
-J & 
\end{pmatrix}, \quad g \in \GL_{2n}.
\end{equation}
Let $H := \Sp_{2n}$ denote the symplectic group, that is, the group of fixed points of $\theta$. We will make the following convention:
\begin{num}
    \item \label{eq:group_convention} For a subgroup $S$ of $G$, we will always denote by $S'$ the intersection $S \cap H$.
\end{num}

Fix the standard maximal split diagonal torus $A_0$ of $G$ and $A_0'$ of $H$ and fix upper triangular Borel subgroups $B'$ of $H$ and $B$ of $G$ respectively.

We will denote by $\cF := \cF^G$ (resp. $\cF' := \cF^H$) be set of semi-standard parabolic subgroups of $G$ and $H$ respectively. 

For positive integers $a_1, \cdots, a_k$, we write $P_{a_1,\cdots,a_k}$ for the standard parabolic subgroup of $\GL_{a_1 + \cdots +a_k}$ of  type $(a_1,\cdots,a_k)$, that is, the parabolic subgroup with Levi subgroup $\GL_{a_1} \times \cdots  \times \GL_{a_k}$.






\section{Absolute convergence}
\label{sec:convergence}

In this section, we show that cuspidal Eisenstein series associated to maximal parabolic subgroups of $\GL_{2n+1}$ are of rapid decay on $\Sp_{2n}$. Indeed they are compactly supported in the function field case. In particular, their $\Sp_{2n}$-periods are absolutely convergent.

Recall that $G := \GL_{2n+1}$ and $H := \Sp_{2n}$. Recall that for a subgroup $S$ of $G$, we denote by $S' := S \cap H$.

\subsection{Function field case}

Let $C$ be a smooth projective and geometrically connected curve over $\F_q$, let $F$ be the function field of $C$ and let $\Bun_G$ denotes the moduli stack of $G$ bundles on $C$.

For convenience, we work in the unramified setting. The modification to the ramified case is obvious (replacing $\Bun_G$ by $\Bun_{G,D}$ for some divisor $D$). 

We recall the Harder-Narasimhan filtration of a symplectic vector bundle. Let $\cE_s$ be a symplectic vector bundle on $C$. By definition, the Harder-Narasimhan filtration of $\cE_s$ is given by
    \begin{equation} \label{eq:HN-symplectic}
     0 \subsetneq \cW_1 \subsetneq \cdots \subsetneq \cW_k \subset \cW_k^{\perp} \subsetneq \cdots \subsetneq \cW_1^{\perp} \subsetneq \cE_s
    \end{equation}
    where each $\cW_i$ is isotropic subbundle of $\cE_s$ with slopes of consecutive quotients are strictly decreasing. (If $\cW_k = \cW_k^{\perp}$, then we only count it once.)

Let $\Bun_H^{\mathrm{gaps}\le 2g-2}(\F_q)$ denote that set of symplectic vector bundles with all gaps in the slope sequences are bounded by $2g-2$. Note that $\Bun_H^{\mathrm{gaps}\le 2g-2}(\F_q)$ is a finite set. 

\begin{proposition} \label{prop:function_field_compact_support}
    Let $f: \Bun_{G}(\F_q) \to \C$ be a function such that the constant terms $f_P = 0$ for any non-maximal parabolic subgroup $P$ of $G$. Then the restriction of $f$ on $\Bun_{H}(\F_q)$ is finitely supported.
\end{proposition}

\begin{proof}
    We show that $f$ vanishes outside $\Bun_{H}^{\opn{gaps} \le 2g-2}(\F_q)$. Let $\cE_s \not \in \Bun_{H}^{\opn{gaps} \le 2g-2}(\F_q)$. Assume that Harder-Narasimhan filtration of $\cE_s$ is given by \eqref{eq:HN-symplectic}.

   The Harder-Narasimhan filtration of the vector bundle $\cE_s \oplus \cO$ is then given by
    \begin{equation*}
        0 \subsetneq \cW_1 \subsetneq \cdots \subsetneq  \cW_k  \subsetneq \cW_k^\perp \oplus \cO \subsetneq \cdots \subsetneq \cW_1^\perp \oplus \cO \subsetneq \cE_s.
    \end{equation*}
    Therefore, the proposition follows from lemma \ref{lem:function_field_compact_support} below.
\end{proof}
 
\begin{lemma} \label{lem:function_field_compact_support}
Let $f:\Bun_{G}(\F_q) \to \C$ satisfies the condition in proposition ~\ref{prop:function_field_compact_support}. Let $\cE \in \Bun_{G}(\F_q)$ such that in the slope sequence of $\cE$ there are two gaps $> 2g-2$. Then $f(\cE)=0$.
\end{lemma}

\begin{proof}
   Suppose that the Harder-Narasimhan filtration of $\cE$ is given by $0 \subsetneq \cW_1 \cdots \subsetneq \cW_k \subsetneq \cE$ with slopes $\mu_r := \mu(\cW_r/\cW_{r-1})$ satisfies $\mu_i - \mu_{i-1}>2g-2, \mu_j-\mu_{j-1}>2g-2$.
   
    Take $a = \opn{rank} \cW_i$, $b = \opn{rank} \cW_j$. Consider the parabolic subgroup $P=P_{a,b,c}=MN$ of $\GL_{2n+1}$. For $(\cV_1,\cV_2,\cV_3) \in \Bun_M(\bF_q)$, we have
    \[
    f_P(\cV_1,\cV_2,\cV_3) = \sum_{ \substack{\cF^{\bullet} := 0 \subset \cF_1 \subset \cF_2 \subset \cF_3 \subset \cV \in \Bun_P(\bF_q) \\ \cF_i/\cF_{i-1} \cong \cV_i }}  \frac{1}{\# \opn{Aut}(\cF^\bullet)} f(\cV) = 0
    \]
    Consider the case when $\cV_1 = \cW_i$, $\cV_2 = \cW_j/\cW_i$, $\cW_3 = \cE/\cW_j$. Since for $i=1,2$, $\mu(\cV_i)>\mu(\cV_{i+1})+2g-2$, we have $\opn{Ext}^1(\cV_2,\cV_1) = \opn{Ext}^1(\cV_3,\cV_2)=0$. Therefore $f_P(\cV_1,\cV_2,\cV_3)$ is a non-zero multiple of $f(\cE)$. This implies $f(\cE)=0$.
\end{proof}

Note that cuspidal Eisenstein series associated to maximal parabolic subgroup satisfy the assumption of the proposition \ref{prop:function_field_compact_support}, therefore we obtain.

\begin{corollary}
    Let $P$ be a maximal parabolic subgroup of $G$, let $\varphi \in \cA_{P,\opn{cusp}}(G)$ be a right $K = G(\mathbb{O})$ invariant function. Then the Eisenstein series $E(\cdot,\varphi,\lambda)$ is compactly supported on $[H]$ whenever $\lambda$ is regular.
\end{corollary}

\begin{remark}
    As previously noted, the assumption that $\varphi$ is unramified is only for sake of convenience. The corollary holds without this assumption.
\end{remark}

\subsection{Number field case}

In the number field case, we use Zydor's truncation to replace the Harder-Narasimhan argument in the function field. The reader can compare the approach here and Lafforgue's interpretation of Arthur truncation in the function field case \cite{Lafforgue}. Let $\lambda: \bG_m \to A_0'$ be a cocharacter, recall that we have a parabolic subgroup $P(\lambda)$ associated to $\lambda$ recalled in ~\eqref{eq:dynamical}. We denote by $\lambda^G$ the cocharacter $i \circ \lambda$, where $i:A_0' \to A_0$ is the natural embedding.

\begin{proposition}
    The map
    \begin{equation*}
        \cF'(B')\to \cF, \quad P(\lambda) \mapsto P(\lambda^G)
    \end{equation*}
    is a well-defined injection.
\end{proposition}

\begin{proof}
    For $\lambda$ of the form
    \begin{equation*}
        \lambda(t) = \opn{diag}(t^{a_1},\cdots,t^{a_n},t^{-a_n},\cdots,t^{-a_1}),
    \end{equation*}
    where $a_1 \ge a_2 \cdots \ge a_n \ge 0$ are integers. For $n+1 \le k \le 2n$, we put $a_k := -a_{2n+1-k}$. Then the Lie algebra of $P(\lambda^G)$ consisting of roots $e_i-e_j$ where $1 \le i \le j \le 2n$ or $a_i=a_j$ or $a_i=0, j=2n+1$ or $a_j=0, i=2n+1$. It is easily seen that these conditions only depends on $P(\lambda)$.
\end{proof}

We denote the image of the map by $\cF^H$, and call its elements the \emph{$H$-relevant parabolic subgroups}.

For $P' \in \cF'(B')$ we denote by $P$ the image of $P'$ in $\cF^H$ in the map above. Note that $P \cap H = P'$, so this notation is compatible with our convention in ~\eqref{eq:group_convention}

We recall the Zydor's construction of truncation operator in ~\cite[\S 3.7]{Zydor19} For $\varphi \in \cT([G])$ and $h \in [H]$, we define
\begin{equation} \label{eq:Zydor_truncation}
    \Lambda^T_H \varphi(h) = \sum_{P' \in \cF'} \varepsilon_{P'} \sum_{\gamma \in P'(F) \backslash H(F)} \widehat{\tau}_{P'}(H_{P'}(\gamma h)-T_{P'}) \varphi_P(\gamma h).
\end{equation}

The following theorem is due to Zydor ~\cite[theorem 3.9]{Zydor19}, for the reader's convenience, we include a proof here.

\begin{theorem}[Zydor] \label{thm:Zydor_truncation_property}
    For $T$ sufficiently positive, $\Lambda^T \varphi \in \cS^0([H])$. More precisely, for any $N>0$, there exists a continuous semi-norm $\| \cdot \|$ (depending on $N$) on $\cT([G])$ such that 
    \begin{equation*}
        \| \Lambda^T_H \varphi \|_{\infty,N} \ll \| \varphi \|.
    \end{equation*}
    Where we recall that $\| \cdot \|_{\infty,N}$ is the norm defined in \eqref{eq:norm_infty_N}.
\end{theorem}

\begin{proof}
    Recall the function space $\cT(H)$ in ~\cite[\S 4.9]{BLX} (denoted by $\cT_{\cF}(G)$ in \emph{loc. cit}). It consists of tuples of functions $({}_{P'} \varphi)$ for each $P' \in \cF'(B')$, such that ${}_{P'}\varphi' \in \cT(P'(F) \backslash H(\bA))$ for any $P'$ and ${}_{P'} \varphi - {}_{Q'} \varphi \in \cS_{d_{P'}^{Q'}}(P'(F) \backslash H(\bA))$ for any $P' \subset Q'$ (see \emph{loc. cit.} for the definition of $d_{P'}^{Q'}$). The space $\cT(H)$ carries a natural topology and by ~\cite[Proposition 4.27]{BLX}, it suffices to check that the tuple $(\varphi_P)_{P' \in \cF'(B')}$ belongs to $\cT(H)$ and the map
    \begin{equation*}
        \cT([G]) \to \cT(H), \quad \varphi \mapsto (\varphi_P)
    \end{equation*}
    is continuous.

    By the approximation by constant term ~\cite[Proposition 2.5.14.1]{BPCZ}, it suffices to check that $d_{P'}^{Q'} \sim d_{P}^Q$ on $[H]_{P'}$. For a suitably chosen Siegel set, we have $\fs^{P'} \subset \fs^{P}$, therefore, it reduces to the following claim:
    \begin{num}
        \item For $P' \subset Q'$, let $\alpha \in X^*(A_0)$ be a root in $\fn_{P} \cap \fm_Q$, then a multiple of $\alpha|_{A_0'}$ is a root in $\fn_{P'} \cap \fm_{Q'}$,
    \end{num}
    which one can easily check directly.
\end{proof}

We record a corollary of the theorem above.

\begin{corollary} \label{cor:Eisenstein_series_rapid_decrease}
    Let $Q$ be a maximal parabolic subgroup of $G$, and let $\varphi \in \cA_{Q,\opn{cusp}}(G)$. Then for $\lambda \in \fa_{P,\C}^*$ such that the Eisenstein series $E(g,\varphi,\lambda)$ is holomorphic, it is Schwartz on $[H]$. In other words, $h \mapsto E(h,\varphi,\lambda) \in \cS([H])$.
\end{corollary}

\begin{proof}
    Note that for any $P' \in \cF'(B')$ such that $P' \ne G'$, the corresponding element $P$ in $\cF^H$ is \emph{not} a maximal parabolic subgroup of $G$. Hence by the constant term formula \eqref{eq:constant_term}, $E_P(h,\varphi,\lambda)=0$. As a consequence, for any sufficiently regular $T \in \fa'_{0}$, $\Lambda^T_H E(h,\varphi,\lambda) = E(h,\varphi,\lambda)$. The corollary the follows from theorem ~\ref{thm:Zydor_truncation_property} (applied to $R(X)E(\cdot,\varphi,\lambda)$ for any $X \in \cU(\fg_{\infty}))$.
\end{proof}

In particular, for any cuspidal Eisenstein series associated to maximal parabolic subgroup, the period
\[
    \int_{[H]} E(h,\varphi,\lambda) \rd h
\]
is absolutely convergent whenever $\lambda$ is regular.

\section{Vanishing of certain periods}
\label{sec:cuspidal}

\subsection{A lemma on Klingen-mirabolic period}


Let $P$ be the subgroup of $\GL_{2n}$ of the following form
\begin{equation*}
    \begin{pmatrix}
        1 & * & * \\
          &  g & * \\
          &          & 1
    \end{pmatrix}, \quad g \in \GL_{2n-2}.
\end{equation*}
We call it the Klingen-mirabolic subgroup of $\GL_{2n}$. Recall $P^\prime=P \cap H$. We call it the Klingen-mirabolic subgroup of $H$. Write the Levi decomposition of $P=\GL_{2n}^\flat U$, where $U$ is the unipotent radical and
\[
    \GL_{2n}^\flat = \begin{pmatrix}
        1 &   &    \\
          & g &  \\
          &       & 1
    \end{pmatrix}, \quad g \in \GL_{2n-2}.
\]
Then $P' = H^\flat U'$, where
\[
    H^\flat = \begin{pmatrix}
        1 &   &    \\
          & h &  \\
          &       & 1
    \end{pmatrix}, \quad h \in \Sp_{2n-2},
\]
and
\[
    U' = \begin{pmatrix}
        1 & a & b & c \\
          & I_{n-1} &  & J_{n-1} {}^tb \\
          &        &   I_{n-1} &  -J_{n-1} {}^ta \\
          &        &           &  1
    \end{pmatrix} ,\quad a, b \in F_{n-1}, c \in F.  
\]
The form of $U^\prime$ follows from the general computation that
\begin{equation*}
    \begin{pmatrix}
        I_{n-a} & X & Y & Z \\
                 & I_{a} &  & W \\
                 &    &  I_{a} & T \\
                 &   &    & I_{n-a}
    \end{pmatrix}
\end{equation*}
is an element of the symplectic group if and only if:
\begin{equation*}
    T = -J_{a}{}^tX J_{n-a}, \quad W = J_{a}{}^tY J_{n-a}, \quad Z^t J_{n-a} + W^t J_{a} T - T^t J_{a} W - J_{n-a}Z = 0.
\end{equation*}

Let $P^\flat$ be the  Klingen-mirabolic subgroup of $\GL_{2n}^\flat$ and let $P^{\prime,\flat}$ be the Klingen-mirabolic subgroup of $H^\flat$.

Note that 
\begin{num}
    \item \label{eq:normality} $U'$ is a normal subgroup of $U$ and $U/U'$ is abelian.
\end{num}

Let $\rd p'$ be the \emph{right} Haar measure on $P'(\bA)$. More concretely, we have
\begin{equation*} \label{eq:integration_on_P'}
  \int_{[P']} f(p') \rd p' = \int_{[H^\flat]} \int_{[U']} f(u'h) \rd u' \rd h. 
\end{equation*}

We say a parabolic subgroup $Q$ of $\GL_{2n}$ is \emph{$\theta$-stable}, if $\theta(Q) = Q$, where we recall $\theta$ is the involution defined in ~\eqref{involution}. We say a smooth function $f$ on $[P]$ is \emph{$\theta$-cuspidal} if, for all $\theta$-stable parabolic subgroup $Q$ of $\GL_{2n}$, the constant term 
\begin{equation*}
    f_Q(p) := \int_{[N_Q]} f(np) \rd n 
\end{equation*}
vanishes. Note that $N_Q \subset P$ so the integration above makes sense. More concretely, we have the following description:
\begin{num}
    \item \label{eq:theta_cuspidal} $f$ is $\theta$-cuspidal, if and only if $f_{N_{k,2n-2k,k}} = 0$ for all $k=1,2,\cdots,n$.
\end{num}

\begin{lemma}[See also \cite{MO21}] \label{lem:Klingen-mirabolic}
    Let $f \in \cS([P])$ be a $\theta$-cuspidal function on $[P]$, then we have
    \[
        \int_{[P']} f(p') \rd p' = 0.
    \]
\end{lemma}

\begin{proof}
    For $p \in P(\bA)$, denote by $f_{U'}$ the function defined by
    \begin{equation*}
        f_{U'}(p) := \int_{[U']} f(u'p) \rd u'.
    \end{equation*}

    By ~\eqref{eq:normality}, for any $p \in P(\bA)$, $p \mapsto f_{U'}(up)$ defines a function on $[U/U']$. 

    Let $\psi_n$ be the character on $[U/U']$ defined by
    \begin{equation*}
        \psi_n(u) = \psi(u_{12} + u_{n-1,n}),
    \end{equation*}
    where $u_{12}$ and $u_{n-1,n}$ denote the $(1,2)$ and $(n-1, n)$ entries of $u$ respectively.
    
    Note $U / U^\prime$ is abelian, by Fourier expansion on $[U/U']$,  we can write
    \begin{equation} \label{eq:vanishing_1}
        f_{U'}(p) =  \sum_{\chi} f_{U',\chi}(p),
    \end{equation}
    where the Fourier coefficient $f_{U',\chi}$ is defined by
    \begin{equation*}
        f_{U',\chi}(p) := \int_{[U/U']} f_{U'}(up) \chi^{-1}(u) \rd u.
    \end{equation*}

    When $\psi = \mathbbm{1}$ is the trivial character, then $f_{U',\mathbbm{1}}(g) = f_U(g) = 0$ by the cuspidality assumption. 

    The nontrivial character on $[U/U']$ can be written as $u \mapsto \psi_n(\gamma^{-1} u \gamma)$, where $\gamma \in H^\flat(F)$. Therefore, the equation ~\eqref{eq:vanishing_1} can be written as
    \[
        f_{U'}(p) = \sum_{\gamma \in P^{\prime,\flat}(F) \backslash H^\flat(F)} f_{U',\psi_n}(\gamma p).
    \]
    Therefore by ~\eqref{eq:integration_on_P'}
    \begin{equation*}
        \int_{[P']} f(p') \rd p' = \int_{P^{\prime,\flat}(F) \backslash H^\flat(\bA)} f_{U',\psi_n}(h) \rd h = \int_{P^{\prime,\flat}(\bA) \backslash H^\flat(\bA)} \int_{[P^{\prime,\flat}]}f_{U',\psi_n}(ph) \rd p \rd h.
    \end{equation*}
    One check easily that $f_{U',\psi_n}$ is a $\theta$-cuspidal function on $P^\flat$. Therefore, by induction on $n$, we only need to check the case when $n=1$, in which case the proposition is trivial.
\end{proof}

\subsection{Vanishing of symplectic period}

For $g \in \GL_{2n}(\bA)$ and $\Phi \in \cS(\bA_n)$, we put
\[
    \Theta'(g,\Phi) = \lvert \det g \rvert^{\frac 12} \sum_{ 0 \neq v \in F_{2n} } \Phi(vg).
\]
Recall the Epstein series 
\begin{equation*}
    E(g,\Phi,s) = \int_{A_{\GL_{2n}}^\infty} \Theta'(ag,\Phi) \lvert \det ag \rvert^{s-\frac 12} \rd a, \quad g \in \GL_{2n}(\bA).
\end{equation*}
By ~\cite[Lemma 4.2]{JacquetShalika}, the integral defining $E(g,\Phi,s)$ is absolutely convergent when $\opn{Re}(s)>1$ and the map $s \mapsto E(\cdot,\Phi,s)$ extends to a meromorphic function valued in $\cT([\GL_{2n}])$ with simple poles at $s=0,1$ of respective residues $\Phi(0)$ and $\widehat{\Phi}(0)$.

\begin{lemma}
    Let $\varphi \in \cT_{\opn{cusp}}([\GL_{2n}])$. Thus, for any $s \ne \{0,1\}$, we have
    \begin{equation} \label{eq:symplectic_theta_period}
        \int_{[\Sp_{2n}]} \varphi(h) E(h,\Phi,s) = 0.
    \end{equation}
\end{lemma}

\begin{proof}
    The integral is absolutely convergent because $\varphi \in \cS([\Sp_{2n}])$.

    Recall that $P^\prime$ denotes the Klingen-mirabolic subgroup of $H=\Sp_{2n}$. Let $e_{2n}=(0, \cdots, 0, 1)$; then $P^\prime$ is the stabilizer of $e_{2n}$ in $\Sp_{2n}$.

    For $\opn{Re}(s)>1$, we put
    \begin{equation*}
        F(g,\Phi,s) = \lvert \det g \rvert^s \int_{A_{\GL_{2n}}^\infty} \Phi(ae_{2n} g) \lvert \det a \rvert^s \rd a.
    \end{equation*}
    The integral is absolutely convergent. Then for $\opn{Re}(s)>1$, we have 
    \begin{equation*}
           E(g,\Phi,s) = \sum_{\gamma \in P'(F) \backslash \Sp_{2n}(F)} F(\gamma g ,\Phi,s).
    \end{equation*}
    Therefore the integral ~\eqref{eq:symplectic_theta_period} can be written as
    \begin{equation*}
        \int_{P'(F) \backslash \Sp_{2n}(\bA)} \varphi(h) F(h,\Phi,s) \rd h. 
    \end{equation*}
    Note that for all $p' \in P'(\bA)$ and $h \in H(\bA)$, we have $F(p'h,\Phi,s) = F(h,\Phi,s)$. The integral above is therefore equal to
    \begin{equation*}
        \int_{P'(\bA) \backslash \Sp_{2n}(\bA)} \int_{[P']} \varphi(p'h) F(h,\Phi,s) \rd p' \rd h,
    \end{equation*}
    which vanishes by lemma ~\ref{lem:Klingen-mirabolic}. Since the integral ~\eqref{eq:symplectic_theta_period} is meromorphic in $s$, we see that it vanishes for any $s \ne \{0,1\}$.
\end{proof}

\begin{corollary} \label{cor:symplectic_period}
    Let $\varphi \in \cT_{\opn{cusp}}([\GL_{2n}])$, then
    \begin{equation*}
        \int_{[\Sp_{2n}]} \varphi(h) \rd h = 0.
    \end{equation*}
\end{corollary}

\begin{proof}
    Take $\Phi \in \cS(\bA_n)$ such that $\Phi(0)=1$. Then, taking residue of ~\eqref{eq:symplectic_theta_period} at $s=0$ suffices. 
\end{proof}

\begin{remark}
    The proposition is also proved in ~\cite{JR}; however, the proof applies to cuspidal automorphic forms, and the proposition above can be applied to any cuspidal function.
\end{remark}

\subsection{Vanishing of certain $\Sp_{2n} \subset \GL_{2n+1}$-period}
For any $f \in \cT([G])$, we define
\[
    \cP_{H}(f)= \int_{[H]}f(h) \rd h
\]
if the integral is convergent.
\begin{proposition} \label{prop:main_vanishing}
    We have $\cP_{H}(f) = 0$ if either
    \begin{itemize}
        \item $f \in \cT_{\opn{cusp}}([G])$, or
        \item $f$ is a pseudo-Eisenstein series $E(g,\varphi)$ for $\varphi \in \cS_{\opn{cusp}}([G]_Q)$, where $Q$ is a maximal parabolic subgroup not associated to $P_{n,n+1}$.
    \end{itemize}
\end{proposition}

\begin{proof}
    By the formula for the constant term of pseudo-Eisenstein series (see ~\eqref{eq:constant_term_pseudo}), in the first case or the second case when $Q$ is not of the type $(1,2n)$ or $(2n,1)$, we have the Fourier expansion
    \begin{equation*}
        f(g) = \sum_{\gamma \in \cP_{2n}(F) \backslash \GL_{2n}(F)} f_{\psi}(\gamma g),
    \end{equation*}
    where $\cP_{2n}$ denotes the mirabolic subgroup of $\GL_{2n}$ fixing $e_{2n}$ and 
    \begin{equation*}
        f_\psi(\gamma g) = \int_{[U_{2n}]} f(ug) \psi^{-1}(u_{2n}) \rd u,
    \end{equation*}
    where the last column of $u$ is given by ${}^t(u_1,\cdots,u_{2n},1)$.

    Note that under the right action of $H(F)$ on $\cP_{2n}(F) \backslash \GL_{2n}(F)$, there is only one orbit, and the stabilizer of $1$ is $P'$. Therefore, we can write
    \begin{equation*}
        \cP_H(f) = \int_{P'(F) \backslash H(\bA)} f_\psi(h) \rd h.
    \end{equation*}
    Using ~\eqref{eq:constant_term_pseudo} again, the map $[P] \ni p  \mapsto f_\psi(p)$ is a $\theta$-cuspidal function on $[P]$. Therefore, the proposition follows from lemma ~\ref{lem:Klingen-mirabolic}.

    We then treat the case when $Q$ is of type $(1,2n)$ or $(2n,1)$. We prove the case when $Q$ is of type $(2n,1)$, the case when $Q$ is of type $(1,2n)$ is the same.

    By ~\eqref{eq:constant_term_pseudo}, we have
    \[
        f(g) =  \sum_{\gamma \in \cP_{2n}(F) \backslash \GL_{2n}(F)} f_{\psi}(\gamma g) + \varphi(g).
    \]
    The argument above shows that
    \begin{equation*}
        \int_{[H]} \sum_{\gamma \in \cP_{2n}(F) \backslash \GL_{2n}(F)} f_{\psi}(\gamma h) \rd h = 0,
    \end{equation*}
    it remains to show
    \begin{equation*}
        \int_{[H]} \varphi(h) \rd h = 0.
    \end{equation*}
    This follows from corollary ~\ref{cor:symplectic_period}.
\end{proof}
\begin{remark} 
    \begin{itemize}
        \item  In theorem ~\ref{thm:main}, we will show that the same conclusion holds after we replace the pseudo-Eisenstein series by the actual Eisenstein series.
        \item  When $f$ is a cuspidal automorphic form, the vanishing is also stated in \cite[Theorem,(2)]{AGR}. In loc. cit., the authors claim that for a generic representation $\pi$ of $\GL_{2n+1}(k)$ over a local field $k$, we have $\Hom_{\Sp_{2n}(k)}(\pi,1) = 0$. This seems incorrect if one believes the local conjecture for this BZSV quadruple. Moreover, the computations in \S \ref{ssec:period_Eisenstein_series} imply that $\Hom_{\Sp_{2n}(\bA)}(\Pi, 1) \neq 0$ for certain global parabolic induction $\Pi$; thus, we have $\Hom_{\Sp_{2n}(F_v)}(\Pi_v, 1) \neq 0$ for each local place $v$.
    \end{itemize}
\end{remark}
We observe that the above proof applies also for $f \in \cS_{\opn{cusp}}([\cP_{2n+1}])$, where $\cP_{2n+1}$ denotes the mirabolic subgroup of $G=\GL_{2n+1}$.

Let $r \geq 1$ and $k \geq 0$, denote by $U_{2r+1,k}$ the subgroup of $N_{2r+1,k}$ consisting of elements of the form
\[
    \begin{pmatrix}
I_{2r} & 0 & * \\
 & 1 & 0 \\
 &  & I_k
\end{pmatrix}.
\]
Identify $\Sp_{2r}$ with the subgroup of $\GL_{2r+1} \subset M_{2r+1,k}$ at the top-left corner as usual. 
\begin{corollary}\label{mixed period}
    For any $f \in \cS_{\opn{cusp}}([\GL_{2r+k+1}])$, we have
    \[
        \int_{[\Sp_{2r} \ltimes U_{2r+1,k}]} f(h) \rd h =0.
    \]
\end{corollary}
\begin{proof}
    For $g \in \GL_{2r+k+1}(\bA)$, denote by $f_{U_{2r+1,k}}$ the function defined by
    \[
        f_{U_{2r+1,k}}(g):=\int_{[U_{2r+1,k}]}f(ug) \rd g.
    \]
    Then one checks easily that $f_{U_{2r+1,k}} \in \cS_{\opn{cusp}}([\cP_{2r+1}])$. By the previous remark, we have
    \[
         \int_{[\Sp_{2r} \ltimes U_{2r+1,k}]} f(h) \rd h= \int_{[\Sp_{2r}]}f_{U_{2r+1,k}} (h) \rd h =0. 
    \]
\end{proof}

\section{\texorpdfstring{Orbits of $\Sp_{2n}$ acting on the Grassmannian}{Orbits of Sp_2n  acting on the Grassmannian}}
\label{sec:orbits}

Throughout \S \ref{sec:orbits} and \S \ref{sec:proof}, we let $P := P_{n,n+1}$ be the maximal parabolic subgroup of type $(n,n+1)$ of $\GL_{2n+1}$ and let $P=M \cdot N$ denote its Levi decompostion. So we do not follow the notations in \S \ref{sec:cuspidal}: we hope that it will cause no confusion.

\subsection{\texorpdfstring{Classification of the orbits of $P(F) \backslash G(F) / H(F)$}{Classification of the orbits P \ G/H }}  

In this subsection, we give a complete classification of the double cosets $P(F) \backslash G(F)/H(F)$, or equivalently, the orbits of the $H(F)$-action on $P(F) \backslash G(F)$. Let $V_0=\langle e_1, \cdots ,e_n \rangle$ be the standard $n$-dimensional subspace. Then the map
$ g \mapsto g^{-1}V_0$ induces a bijection between $P(F) \backslash G(F)$ and the set $\Gr_n:= \Gr(n,2n+1)(F)$ of $n$-dimensional subspaces of $F^{2n+1}$. So the double coset $P(F) \backslash G(F) / H(F)$ can be identified with the orbits of $\Gr_n(F)$ under the action of $\Sp_{2n}(F)$. The orbits can be grouped into 4 different types, which we call type \RNum{1}, \RNum{2}, \RNum{3}, \RNum{4} respectively. For each $\gamma \in P(F) \backslash G(F)$, we denote by $H_{\gamma}$ the stabilizer of $\gamma$ in $H(F)$.

Recall that $F^{2n}$ denotes the subspace of $F^{2n+1}$ with the last coordinate 0. Let $p:F^{2n+1} \to F^{2n}$ denote the projection to the first $2n$ coordinates. For a subspace $V \subset F^{2n+1}$, we denote by $i(V):= V \cap F^{2n}$ and denote by $p(V)$ the image of $V$ under $p$. Let $W$ be a subspace of $F^{2n}$, we denote by $r(W)$ the Witt index of $W$.

\begin{enumerate}
    \item(Type \RNum{1} orbit) $V \subset F^{2n}$, then by Witt's theorem, the Witt index $r:=r(V)$ is the only invariant of the $\Sp_{2n}(F)$ orbit of $V$ in $\Gr_n(F)$. In other words, if $r(V)=r(W)$ for $W \in \Gr_n(F)$, then there exists $h \in \Sp_{2n}(F)$ such that $h^{-1}V=W$. It's clear $r \in \{ 0,1,\cdots,[\frac n2] \}$

    \item(Type \RNum{2} orbit) $e_{2n+1} \in V$. Then by Witt's theorem again, the Witt index $r:=r(i(V))=r(p(V))$ is the only invariant of the $\Sp_{2n}(F)$ orbit of $V$ in $\Gr_n(F)$. In this case, $r \in \{0,1,\cdots,[\frac{n-1}{2}]\}$ 

    \item(Type \RNum{3} and \RNum{4} orbits) $e_{2n+1} \not \in V$ and $V \not \subset F^{2n}$. In this case, we have $\dim p(V) = \dim i(V) + 1$. Type \RNum{3} and \RNum{4} orbits correspond to whether $r(p(V)) = r(i(V))$ or $r(p(V)) = r(i(V))+1$ respectively. Consider the flag $0 \subset i(V) \subset p(V) \subset V$, it is easy to apply the next lemma to see that
        \begin{itemize}
            \item For $V$ as above. The subspace $V$ with $r(p(V)) = r(i(V)) = r$ form a single $\Sp_{2n}(F)$ orbit. The possible values of $r$ are the set $\{0,1,\cdots,[\frac{n-1}{2}]\}$. We call this a type \RNum{3} orbit.

            \item Similarly, for $V$ with $r(p(V)) = r(i(V)) + 1 = r$ form a single $\Sp_{2n}(F)$ orbit. The possible values of $r$ are the set $\{1,2,\cdots,[\frac{n}{2}]\}$. We call this a type \RNum{4} orbit. 
        \end{itemize}
    
\end{enumerate}
The following lemma is a basic fact in linear algebra, see e.g. ~\cite{Shmelev}.

Recall that a Darboux basis for a non-degenerate symplectic vector space of dimension $2n$ is a basis $\langle e_1,\cdots,e_n,f_1,\cdots,f_n \rangle$ such that $\langle e_i,e_j \rangle = \langle f_i,f_k \rangle = 0$ and $\langle e_i,f_j \rangle = \delta_{ij}$.  

\begin{lemma} \label{lem:Darboux_basis}
    Let $V$ be a symplectic space over a field $F$ with $\opn{char}(F) \ne 2$. Let $0 \subset V_1 \subset \cdots V_{k} \subset V$ be a flag of vector subspaces of $V$. Then there exists a Darboux basis $\cB = \{e_1,\cdots,e_n,f_1,\cdots,f_n\}$ such that each $V_i$ is generated by a subset of $\cB$. 
\end{lemma}

\subsection{Representatives and stabilizers} \label{ssec:list_of_orbit}
Throughout this section, we write an element $h \in H_\gamma$ as a $2n \times 2n$-matrix to emphasize that it lies in the stabilizer under the action of $H=\Sp_{2n}$. We write the element $\gamma h \gamma^{-1}$ as a $(2n+1)\times (2n+1)$-matrix to make it easy to see its factors in the Levi decomposition of $P$.

\subsubsection{Type \RNum{1} orbits}


Consider type \RNum{1} orbits, for any $0 \leq r \leq [\frac{n}{2}]$, we choose a representative
\[
    V_\gamma=\langle e_1, \cdots, e_{n-2r}, e_{n-r+1}, \cdots, e_n, e_{n+1}, \cdots , e_{n+r}\rangle,
\]
and
\begin{align*}
    \gamma^{-1}=\begin{pmatrix}
I_{n-2r} &  &  &  &  \\
 &  &  & I_r &  \\
 &  & I_r &  &  \\
 & I_r &  &  &  \\
 &  &  &  & I_{n-r+1}
\end{pmatrix}
\end{align*}
such that $\gamma^{-1}V=V_\gamma$. Then $\gamma=\gamma^{-1}$.

\begin{num}
\item \label{eq:main_orbit} If $r=0$, then it is clear that $H_{\gamma}=P^\prime$, the Siegel parabolic subgroup fixing the maximal isotropic subspace $V$. We call this orbit the \emph{main orbit}.
\end{num}

If $r>0$, consider the isotropic subspace $\langle e_1 \cdots, e_{n-2r}\rangle$ again and consider the associated parabolic subgroup $Q_{\gamma}=M_{\gamma} \cdot N_{\gamma}$. Again we have
\[  
    H_{\gamma}=(H_{\gamma} \cap M_{\gamma}) \cdot N_{\gamma}.
\]
Denote by $U_{\gamma}$ the normal subgroup of $N_{\gamma}$ consisting of elements of the following form
\[
    h=\begin{pmatrix}
I_{n-2r} & X &  &  & Y & Z \\
 & I_r &  &  &  & J_r {}^tYJ_{n-2r} \\
 &  & I_r &  &  &  \\
 &  &  & I_r &  &  \\
 &  &  &  & I_r & -J_r{}^tXJ_{n-2r} \\
 &  &  &  &  & I_{n-2r}
\end{pmatrix}.
\]
One can easily compute
\[
    \gamma h \gamma^{-1}=\begin{pmatrix}
I_{n-2r} &  & X & Y & Z &  \\
 & I_{2r} &  &  &  &  \\
 &  & I_r &  & J_r{}^tYJ_{n-2r} &  \\
 &  &  & I_r & -J_r{}^tXJ_{n-2r} &  \\
 &  &  &  & I_{n-2r} &  \\
 &  &  &  &  & 1
\end{pmatrix}.
\]

On the other hand, we have
\[
    M_\gamma=\GL(\langle e_1, \cdots, e_{n-2r}\rangle) \times \Sp(\langle e_{n-2r+1}, \cdots, e_n, e_{n+1}, \cdots, e_{n+2r}\rangle).
\]
It's clear $\GL(\langle e_1, \cdots, e_{n-2r}\rangle) \subset H_{\gamma}$. So we have
\[
    H_{\gamma} \cap M_{\gamma}=\GL(\langle e_1, \cdots, e_{n-2r}\rangle) \times (H_{\gamma} \cap \Sp(\langle e_{n-2r+1}, \cdots, e_n, e_{n+1}, \cdots, e_{n+2r}\rangle) ).
\]
It's clear $H_{\gamma} \cap \Sp(\langle e_{n-2r+1}, \cdots, e_n, e_{n+1}, \cdots, e_{n+2r}\rangle)$ stabilizes $\langle e_{n-r+1}, \cdots, e_n, e_{n+1}, \cdots, e_{n+r} \rangle$. Then it follows
\begin{align*}
    &H_{\gamma} \cap \Sp(\langle e_{n-2r+1}, \cdots, e_n, e_{n+1}, \cdots, e_{n+2r}\rangle) \\
    &= \Sp(\langle e_{n-2r+1}, \cdots, e_{n-r}, e_{n+r+1}, \cdots, e_{n+2r} \rangle) \times \Sp(\langle e_{n-r+1}, \cdots, e_n, e_{n+1}, \cdots, e_{n+r} \rangle).
\end{align*}

Denote the subgroup $\Sp(\langle e_{n-2r+1}, \cdots, e_{n-r}, e_{n+r+1}, \cdots, e_{n+2r}\rangle)$ by $S_\gamma$. Then $S_\gamma$ is normal in $H_\gamma \cap M_\gamma$. It's easy to see $S_\gamma$ normalizes $U_\gamma$; thus, $S_\gamma \cdot U_\gamma$ is a subgroup of $H_{\gamma}$. Also, by direct computation, we can show that
\[
    S_\gamma \cdot U_\gamma \lhd H_{\gamma}=(H_{\gamma} \cap M_{\gamma}) \cdot N_{\gamma}.
\]
Write elements of $S_\gamma$ in the following form
\[
    h=\begin{pmatrix}
I_{n-2r} &  &  &  &  \\
 & A &  & B &  \\
 &  & I_{2r} &  &  \\
 & C &  & D &  \\
 &  &  &  & I_{n-2r}
\end{pmatrix}.
\]
One can easily compute
\[
    \gamma h \gamma^{-1}=\begin{pmatrix}
I_n &  &  &  \\
 & A & B &  \\
 & C & D &  \\
 &  &  & I_{n-2r+1}
\end{pmatrix}.
\]


\subsubsection{Type \RNum{2} orbits}

Consider type \RNum{2} orbits, for any $0 \leq r \leq [\frac{n-1}{2}]$, we choose a representative
\[
    V_\gamma=\langle e_2, \cdots, e_{n-2r}, e_{n-r+1}, \cdots, e_n, e_{n+1}, \cdots , e_{n+r}, e_{2n+1}\rangle,
\]
and
\begin{align*}
\gamma^{-1}= \begin{pmatrix}
 &  &  &  &  &  & 1 \\
 & I_{n-2r-1} &  &  &  &  &  \\
 &  &  &  & I_r &  &  \\
 &  &  & I_r &  &  &  \\
 &  & I_r &  &  &  &  \\
 &  &  &  &  & I_{n-r} &  \\
1 &  &  &  &  &  & 
\end{pmatrix}
\end{align*}
such that $\gamma^{-1}V=V_\gamma$. Then $\gamma=\gamma^{-1}$.

If $r>0$, denote by $U_{\gamma}$ the normal subgroup of $N_{\gamma}$ consisting of elements of the following form
\begin{align*}
    h=\begin{pmatrix}
1 &  &  &  &  &  &  &  \\
  & I_{n-2r-1} &  & X & Y &   & Z &  \\
  &  & I_r &  &  &   &   &  \\
 &   &  & I_r &  &   & J_r{}^tYJ_{n-2r-1} &  \\
 &  &  &  & I_r &  & -J_r{}^tXJ_{n-2r-1} &  \\
 &  &  &  &  & I_r &  &  \\
 &  &  &  &  &  & I_{n-2r-1} &  \\
 &  &  &  &  &  &  & 1
\end{pmatrix}.
\end{align*}
One can easily compute
\begin{align*}
    \gamma h \gamma^{-1}=\begin{pmatrix}
1 &  &  &  &  &  &  \\
 & I_{n-2r-1} & Y & X &  & Z &  \\
 &  & I_r &  &  & -J_r{}^tXJ_{n-2r-1} &  \\
 &  &  & I_r &  & J_r{}^tYJ_{n-2r-1} &  \\
 &  &  &  & I_{2r} &  &  \\
 &  &  &  &  & I_{n-2r-1} &  \\
 &  &  &  &  &  & I_2
\end{pmatrix}.
\end{align*}
On the other hand, we have
\[
    M_\gamma=\GL(\langle e_2, \cdots, e_{n-2r}\rangle) \times \Sp(\langle e_1, e_{n-2r+1}, \cdots, e_n, e_{n+1}, \cdots, e_{n+2r}, e_{2n}\rangle).
\]
It's clear $\GL(\langle e_2, \cdots, e_{n-2r}\rangle) \subset H_{\gamma}$. So we have
\[
    H_{\gamma} \cap M_{\gamma}=\GL(\langle e_2, \cdots, e_{n-2r}\rangle) \times (H_{\gamma} \cap \Sp(\langle e_1, e_{n-2r+1}, \cdots, e_n, e_{n+1}, \cdots, e_{n+2r}, e_{2n}\rangle) ).
\]
It's clear $H_{\gamma} \cap \Sp(\langle e_1, e_{n-2r+1}, \cdots, e_n, e_{n+1}, \cdots, e_{n+2r}, e_{2n}\rangle)$ stabilizes $\langle e_{n-r+1}, \cdots, e_n, e_{n+1}, \cdots, e_{n+r} \rangle$. Then it follows
\begin{align*}
    &H_{\gamma} \cap \Sp(\langle e_1, e_{n-2r+1}, \cdots, e_n, e_{n+1}, \cdots, e_{n+2r}, e_{2n}\rangle)\\
    &=\Sp(\langle e_{n-r+1}, \cdots, e_n, e_{n+1}, \cdots, e_{n+r} \rangle) \times \Sp(\langle e_1, e_{n-2r+1}, \cdots, e_{n-r}, e_{n+r+1}, \cdots, e_{n+2r},e_{2n}).
\end{align*}
Denote the subgroup $\Sp(\langle e_{n-r+1}, \cdots, e_n, e_{n+1}, \cdots, e_{n+r} \rangle)$ by $S_\gamma$. Then we can show again
\[
    S_\gamma \cdot U_\gamma \lhd H_\gamma=(H_\gamma \cap M_\gamma) \cdot N_\gamma.
\]
Write elements of $S_\gamma$ in the following form
\[
    h=\begin{pmatrix}
I_{n-r} &  &  &  \\
 & A & B &  \\
 & C & D &  \\
 &  &  & I_{n-r}
\end{pmatrix}.
\]
One can easily compute
\[
    \gamma h \gamma^{-1}=\begin{pmatrix}
I_{n-2r} &  &  &  \\
 & D & C &  \\
 & B & A &  \\
 &  &  & I_{n+1}
\end{pmatrix}.
\]

If $r=0$, we let $S_\gamma=1$ and $U_\gamma=N_\gamma$. Write elements in $N_\gamma$ in the following form
\[
    h=\begin{pmatrix}
1 &  & {}^tyJ_{n-1} &  \\
x & I_{n-1} & Z & y \\
 &  & I_{n-1} &  \\
 &  & -{}^txJ_{n-1} & 1
\end{pmatrix}.
\]
One can easily compute
\[
    \gamma h \gamma^{-1}=\begin{pmatrix}
1 &  &  &  &  \\
 & I_{n-1} & Z & y & x \\
 &  & I_{n-1} &  &  \\
 &  & -{}^txJ_{n-1} & 1 &  \\
 &  & {}^tyJ_{n-1} &  & 1
\end{pmatrix}.
\]
\subsubsection{Type \RNum{3} orbits}
Consider type \RNum{3} orbits, for any $0 \leq r \leq [\frac{n-1}{2}]$, we choose a representative
\[
    V_\gamma=\langle e_1+ e_{2n+1}, e_2, \cdots, e_{n-2r}, e_{n-r+1}, \cdots, e_n, e_{n+1}, \cdots , e_{n+r}\rangle,
\]
and
\begin{align*}
    \gamma^{-1}=\begin{pmatrix}
1 &  &  &  &  &  &  \\
 & I_{n-2r-1} &  &  &  &  &  \\
 &  &  &  & I_r &  &  \\
 &  &  & I_r &  &  &  \\
 &  & I_r &  &  &  &  \\
 &  &  &  &  & I_{n-r} &  \\
1 &  &  &  &  &  & 1
\end{pmatrix}
\end{align*}
such that $\gamma^{-1}V=V_\gamma$. Then
\begin{align*}
    \gamma=\begin{pmatrix}
1 &  &  &  &  &  &  \\
 & I_{n-2r-1} &  &  &  &  &  \\
 &  &  &  & I_r &  &  \\
 &  &  & I_r &  &  &  \\
 &  & I_r &  &  &  &  \\
 &  &  &  &  & I_{n-r} &  \\
-1 &  &  &  &  &  & 1
\end{pmatrix}.
\end{align*}

Consider the isotropic subspace $\langle e_2 \cdots, e_{n-2r}\rangle$ again and consider the associated parabolic subgroup $Q_{\gamma}=M_{\gamma} \cdot N_{\gamma}$. Again we have
\[  
    H_{\gamma}=(H_{\gamma} \cap M_{\gamma}) \cdot N_{\gamma}.
\]
If $r>0$, denote by $U_{\gamma}$ the normal subgroup of $N_{\gamma}$ consisting of elements of the following form
\begin{align*}
    h=\begin{pmatrix}
1 &  &  &  &  &  &  &  \\
  & I_{n-2r-1} &  & X & Y &   & Z &  \\
  &  & I_r &  &  &   &   &  \\
 &   &  & I_r &  &   & J_r{}^tYJ_{n-2r-1} &  \\
 &  &  &  & I_r &  & -J_r{}^tXJ_{n-2r-1} &  \\
 &  &  &  &  & I_r &  &  \\
 &  &  &  &  &  & I_{n-2r-1} &  \\
 &  &  &  &  &  &  & 1
\end{pmatrix}.
\end{align*}
One can easily compute
\begin{align*}
    \gamma h \gamma^{-1}=\begin{pmatrix}
1 &  &  &  &  &  &  \\
 & I_{n-2r-1} & Y & X &  & Z &  \\
 &  & I_r &  &  & -J_r{}^tXJ_{n-2r-1} &  \\
 &  &  & I_r &  & J_r{}^tYJ_{n-2r-1} &  \\
 &  &  &  & I_{2r} &  &  \\
 &  &  &  &  & I_{n-2r-1} &  \\
 &  &  &  &  &  & I_2
\end{pmatrix}.
\end{align*}
On the other hand, we have
\[
    M_\gamma=\GL(\langle e_2, \cdots, e_{n-2r}\rangle) \times \Sp(\langle e_1, e_{n-2r+1}, \cdots, e_n, e_{n+1}, \cdots, e_{n+2r}, e_{2n}\rangle).
\]
It's clear $\GL(\langle e_2, \cdots, e_{n-2r}\rangle) \subset H_{\gamma}$. So we have
\[
    H_{\gamma} \cap M_{\gamma}=\GL(\langle e_2, \cdots, e_{n-2r}\rangle) \times (H_{\gamma} \cap \Sp(\langle e_1, e_{n-2r+1}, \cdots, e_n, e_{n+1}, \cdots, e_{n+2r}, e_{2n}\rangle) ).
\]
Note the subgroup $H_{\gamma} \cap \Sp(\langle e_1, e_{n-2r+1}, \cdots, e_n, e_{n+1}, \cdots, e_{n+2r}, e_{2n}\rangle)$ stabilizes the subspace $\langle e_1+e_{2n+1}, e_{n-r+1}, \cdots, e_n, e_{n+1}, \cdots, e_{n+r} \rangle$. Then it also stabilizes $\langle e_{n-r+1}, \cdots, e_n, e_{n+1}, \cdots, e_{n+r} \rangle$, which is the intersection of $\langle e_1+e_{2n+1}, e_{n-r+1}, \cdots, e_n, e_{n+1}, \cdots, e_{n+r} \rangle$ with $F^{2r}$. By similar arguments, we can show
\begin{align*}
    &H_{\gamma} \cap \Sp(\langle e_1, e_{n-2r+1}, \cdots, e_n, e_{n+1}, \cdots, e_{n+2r}, e_{2n}\rangle)\\
    &=\Sp(\langle e_{n-r+1}, \cdots, e_n, e_{n+1}, \cdots, e_{n+r} \rangle) \times (H_\gamma \cap \Sp(\langle e_1, e_{n-2r+1}, \cdots, e_{n-r}, e_{n+r+1}, \cdots, e_{n+2r},e_{2n})).
\end{align*}
Denote the subgroup $\Sp(\langle e_{n-r+1}, \cdots, e_n, e_{n+1}, \cdots, e_{n+r} \rangle)$ by $S_\gamma$. Then $S_\gamma$ is normal in $H_{\gamma} \cap M_{\gamma}$. It's easy to see $S_\gamma$ normalizes $U_\gamma$; thus, $S_\gamma \cdot U_\gamma$ is a subgroup of $H_{\gamma}$. Also, by direct computation, we can show that
\[
    S_\gamma \cdot U_\gamma \lhd H_\gamma=(H_\gamma \cap M_\gamma) \cdot N_\gamma.
\]
Write elements of $S_\gamma$ in the following form
\[
    h=\begin{pmatrix}
I_{n-r} &  &  &  \\
 & A & B &  \\
 & C & D &  \\
 &  &  & I_{n-r}
\end{pmatrix}.
\]
One computes easily
\[
    \gamma h \gamma^{-1}=\begin{pmatrix}
I_{n-2r} &  &  &  \\
 & D & C &  \\
 & B & A &  \\
 &  &  & I_{n+1}
\end{pmatrix}.
\]

If $r=0$, denote by $U_\gamma$ the normal subgroup of $N_\gamma$ consisting elements of the following form
\[
    h=\begin{pmatrix}
1 &  & {}^ty J_{n-1} &  \\
 & I_{n-1} & Z & y \\
 &  & I_{n-1} &   \\
 &  &  & 1
\end{pmatrix}.
\]
One computes easily
\[
    \gamma h \gamma^{-1}=\begin{pmatrix}
1 &  & {}^ty J_{n-1} &  &  \\
 & I_{n-1} & Z & y &  \\
 &  & I_{n-1} &   &  \\
 &  &  & 1 &  \\
 &  & -{}^ty J_{n-1} &  & 1
\end{pmatrix}.
\]
Again, one obtains
\[
    H_\gamma \cap M_\gamma=\GL(\langle e_2, \cdots, e_n \rangle) \times (H_\gamma \cap \Sp(\langle e_1, e_{2n} \rangle)).
\]
This implies the subgroup $S_\gamma=H_\gamma \cap \Sp(\langle e_1, e_{2n} \rangle)$ is normal in $H_\gamma \cap M_\gamma$. And we have again
\[
    S_\gamma \cdot U_\gamma \lhd H_\gamma=(H_\gamma \cap M_\gamma) \cdot N_\gamma.
\]
Write elements of $S_\gamma$ in the following form
\[
    h=\begin{pmatrix}
1 &  & s \\
 & I_{2n-2} &  \\
 &  & 1
\end{pmatrix}.
\]
One computes easily
\[
    \gamma h \gamma^{-1}=\begin{pmatrix}
1 &  & s &  \\
 & I_{2n-2} &  &  \\
 &  & 1 &  \\
 &  & -s & 1
\end{pmatrix}.
\]
\subsubsection{Type \RNum{4} orbits}
Consider type \RNum{4} orbits, for any $1 \leq r \leq [\frac{n}{2}]$, we choose a representative
\[
    V_\gamma=\langle e_1, \cdots, e_{n-2r}, e_{n-r+1}, \cdots, e_n, e_{n+1}+e_{2n+1}, \cdots , e_{n+r}\rangle,
\]
and
\begin{align*}
    \gamma^{-1}=\begin{pmatrix}
I_{n-2r} &  &  &  &  &  &  \\
 &  &  &  & I_r &  &  \\
 &  &  & I_r &  &  &  \\
 & 1 &  &  &  &  &  \\
 &  & I_{r-1} &  &  &  &  \\
 &  &  &  &  & I_{n-r} &  \\
 & 1 &  &  &  &  & 1
\end{pmatrix}
\end{align*}
such that $\gamma^{-1}V=V_\gamma$. Then
\[
    \gamma=\begin{pmatrix}
I_{n-2r} &  &  &  &  &  &  \\
 &  &  & 1 &  &  &  \\
 &  &  &  & I_{r-1} &  &  \\
 &  & I_r &  &  &  &  \\
 & I_r &  &  &  &  &  \\
 &  &  &  &  & I_{n-r} &  \\
 &  &  & -1 &  &  & 1
\end{pmatrix}.
\]

Consider the isotropic subspace $\langle e_1 \cdots, e_{n-2r}\rangle$ again and consider the associated parabolic subgroup $Q_{\gamma}=M_{\gamma} \cdot N_{\gamma}$. Again we have
\[  
    H_{\gamma}=(H_{\gamma} \cap M_{\gamma}) \cdot N_{\gamma}.
\]
Denote by $U_{\gamma}$ the normal subgroup of $N_{\gamma}$ consisting of elements of the following form
\[
    h=\begin{pmatrix}
I_{n-2r} & X &  &  & Y & Z \\
 & I_r &  &  &  & J_r {}^tYJ_{n-2r} \\
 &  & I_r &  &  &  \\
 &  &  & I_r &  &  \\
 &  &  &  & I_r & -J_r{}^tXJ_{n-2r} \\
 &  &  &  &  & I_{n-2r}
\end{pmatrix}
\].
One computes easily
\[
    \gamma h \gamma^{-1}=\begin{pmatrix}
I_{n-2r} &  & X & Y & Z &  \\
 & I_{2r} &  &  &  &  \\
 &  & I_r &  & J_r{}^tYJ_{n-2r} &  \\
 &  &  & I_r & -J_r{}^tXJ_{n-2r} &  \\
 &  &  &  & I_{n-2r} &  \\
 &  &  &  &  & 1
\end{pmatrix}.
\]

On the other hand, we have
\[
    M_\gamma=\GL(\langle e_1, \cdots, e_{n-2r}\rangle) \times \Sp(\langle e_{n-2r+1}, \cdots, e_n, e_{n+1}, \cdots, e_{n+2r}\rangle).
\]
It's clear $\GL(\langle e_1, \cdots, e_{n-2r}\rangle) \subset H_{\gamma}$. So we have
\[
    H_{\gamma} \cap M_{\gamma}=\GL(\langle e_1, \cdots, e_{n-2r}\rangle) \times (H_{\gamma} \cap \Sp(\langle e_{n-2r+1}, \cdots, e_n, e_{n+1}, \cdots, e_{n+2r}\rangle) ).
\]
Similarly, $H_{\gamma} \cap \Sp(\langle e_{n-2r+1}, \cdots, e_n, e_{n+1}, \cdots, e_{n+2r}\rangle)$ stabilizes $\langle e_{n-r+1}, \cdots, e_n, e_{n+1}, \cdots, e_{n+r} \rangle$, which is the projection of $\langle e_{n-r+1}, \cdots, e_n, e_{n+1}+e_{2n+1}, \cdots, e_{n+r} \rangle$ to $F^{2r}$. Then it follows
\begin{align*}
    &H_{\gamma} \cap \Sp(\langle e_{n-2r+1}, \cdots, e_n, e_{n+1}, \cdots, e_{n+2r}\rangle) \\
    &= \Sp(\langle e_{n-2r+1}, \cdots, e_{n-r}, e_{n+r+1}, \cdots, e_{n+2r}) \times (H_\gamma \cap \Sp(\langle e_{n-r+1}, \cdots, e_n, e_{n+1}, \cdots, e_{n+r} \rangle)).
\end{align*}

This implies the subgroup $S_\gamma=\Sp(\langle e_{n-2r+1}, \cdots, e_{n-r}, e_{n+r+1}, \cdots, e_{n+2r}\rangle)$ is normal in $H_{\gamma} \cap M_{\gamma}$. Again, we have
\[
    S_\gamma \cdot U_\gamma \lhd H_{\gamma}=(H_{\gamma} \cap M_{\gamma}) \cdot N_{\gamma}.
\]
Write elements of $S_\gamma$ in the following form
\[
    h=\begin{pmatrix}
I_{n-2r} &  &  &  &  \\
 & A &  & B &  \\
 &  & I_{2r} &  &  \\
 & C &  & D &  \\
 &  &  &  & I_{n-2r}
\end{pmatrix}.
\]
One computes easily
\[
    \gamma h \gamma^{-1}=\begin{pmatrix}
I_n &  &  &  \\
 & A & B &  \\
 & C & D &  \\
 &  &  & I_{n-2r+1}
\end{pmatrix}.
\]

\section{Proof of the main result}
\label{sec:proof}

In this section, we prove our main results (theorem \ref{thm:main}) concerning $\Sp_{2n}$-periods of cusp forms and cuspidal Eisenstein series attached to maximal parabolic subgroups. For convenience, we only treat the number field case; the function field case follows from a similar (indeed easier) argument. Recall that $P = P_{n,n+1}$ is the maximal parabolic subgroup of type $(n,n+1)$ of $\GL_{2n+1}$.

\subsection{Intertwining period}

For $\varphi \in \cT([G]_{P})$ and $\lambda \in \fa^*_{P,\C}$, if convergent, we define the intertwining period of $\varphi$ as
\begin{equation*}
    J(\varphi,\lambda) = \int_{[H]_{P'}} \varphi(h) e^{\langle \lambda, H_{P}(h) \rangle} \rd h.
\end{equation*}
We put $J(\varphi) := J(\varphi,0)$.

\begin{lemma} \label{lem:intertwining_convergence}
    If $\varphi \in \cT_{\opn{cusp}}([G]_{P})$, then integral defining $J(\varphi,\lambda)$ is absolutely convergent for any $\lambda \in \fa_{P,\C}^*$.
\end{lemma}

\begin{proof}
    Replacing $\varphi(h)$ by $\varphi \cdot e^{\langle \lambda,H_{P}(\cdot) \rangle}$, we can assume $\lambda = 0$. Since $\varphi \in \cT_{\opn{cusp}}([G]_{P})$, $\varphi|_{[H]_{P'}} \in \cS([H]_{P'})$ (see \eqref{eq:cusp_form_rapid_decay}) the result follows.
\end{proof}

We can also relate the intertwining period to Whittaker functions. Denote the upper triangular unipotent subgroups of $\GL_i$ by $N_i$. Let $\psi_{n,n+1}$ be the character on $[N_n] \times [N_{n+1}]$ defined by
\begin{equation*}
    \psi_{n,n+1}(u) = \psi(u_{1,2}+\cdots+u_{n-1,n}+u_{n+1,n+2}+\cdots+u_{2n,2n+1}).
\end{equation*}
For $\varphi \in \cT([G]_{P})$, we define 
\begin{equation*}
    W_\varphi(g) = \int_{[N_n] \times [N_{n+1}]} \varphi(ug) \psi_{n,n+1}(u) \rd u. 
\end{equation*}

By the usual expansion into Whittaker functions, recall that $N_{2n+1}' := N_{2n+1} \cap H$, we can write
\begin{equation} \label{eq:J_and_Whittaker}
    \begin{split}
    J(\varphi,\lambda) &= \int_{P^\prime(\bA) \backslash H(\bA)} \int_{N_n(\bA) \backslash \GL_n(\bA)} W_{R(h)\varphi} \begin{pmatrix}
        J{}^tm^{-1}J &  & \\ & m &  \\ & & 1
    \end{pmatrix} \lvert \det m \rvert^{n+1+s_\lambda} e^{\langle \lambda, H_{P}(h) \rangle} \rd m \rd h  \\
    &= \int_{N_{2n+1}'(\bA) \backslash H(\bA)} W_\varphi(h) e^{\langle \lambda, H_P(h) \rangle} \rd h,
    \end{split}
\end{equation}
when $\mathrm{Re}(s_\lambda) \gg 0$, where $s_\lambda \in \C$ is determined by
\[
    \exp \left(\langle \lambda, H_{P}\begin{pmatrix}
        J{}^tm^{-1}J &  & \\ & m &  \\ & & 1
    \end{pmatrix}\rangle \right)= \lvert \det m \rvert ^{s_\lambda}.
\]

\subsection{Period of pseudo-Eisenstein series}

\begin{proposition} \label{prop:pseudo}
    Let $\varphi \in \cS([G]_{P})$ and let $E(\cdot,\varphi)$ be the pseudo-Eisenstein series. Then
    \begin{equation*}
        \int_{[H]} E(h,\varphi) \rd h = J(\varphi)
    \end{equation*}
\end{proposition}

\begin{proof}
    Since $\varphi \in \cS([G]_{P})$, we have $\lvert \varphi \rvert \in \cS^{00}([G]_{P})$. Hence the expression
    \begin{equation*}
        \int_{[H]} \sum_{\gamma \in P(F) \backslash G(F)} \lvert \varphi(\gamma h) \rvert \rd h
    \end{equation*}
    is finite. Let $[P(F) \backslash G(F) / H(F)]$ denote an arbitrary representative of the double coset $P(F) \backslash G(F) / H(F)$. Therefore
    \begin{equation*}
        \begin{split}
              \int_{[H]} E(h,\varphi) \rd h &= \sum_{\gamma \in [P(F) \backslash G(F)/H(F)]} \int_{H_{\gamma}(F) \backslash H(\bA)} \varphi(\gamma h) \rd h \\ 
             =& \sum_{\gamma \in [P(F) \backslash G(F)/H(F)]} \int_{H_\gamma(\bA) \backslash H(\bA)} \int_{[H_\gamma]} \varphi(\gamma h \gamma^{-1} \gamma g) \rd h \rd g
        \end{split}
    \end{equation*}
    Note that $\gamma h \gamma^{-1} \in P$. In \S ~\ref{ssec:list_of_orbit}, we give a complete list of one choice of $[P(F)\backslash G(F)/H(F)]$. When the orbit is not the main orbit (see ~\eqref{eq:main_orbit}), we list a normal subgroup $S_\gamma \cdot U_\gamma$ and compute the explicit form of $\gamma h \gamma^{-1}$ for $h$ in $S_\gamma$ or $U_\gamma$ respectively. Using cuspidality or corollary ~\ref{mixed period}, the integral on this normal subgroup already vanishes. Therefore, only the main orbit contributes, which gives
    \begin{equation*}
        \int_{P'(F) \backslash H(\bA)} \varphi(h) \rd h = J(\varphi).
    \end{equation*}
\end{proof}

\subsection{Period of cuspidal Eisenstein series}
\label{ssec:period_Eisenstein_series}

In this subsection, we prove the main theorem.

\begin{theorem} \label{thm:main}
    Let $Q$ be a maximal parabolic subgroup of $G$ and let $\varphi \in \cA_{Q,\opn{cusp}}(G)$. Then
    \begin{enumerate}
        \item If $Q$ is not associate to $P$, then $\cP_H(E(\cdot,\varphi,\lambda)) = 0$ for any $\lambda$ such that $E(\cdot,\varphi,\lambda)$ is regular.
        \item If $Q=P$, then for any $\lambda$ such that $E(\cdot,\varphi,\lambda)$ is regular,
        \begin{equation*}
            \cP_H(E(\cdot,\varphi,\lambda)) = J(\varphi,\lambda).
        \end{equation*}
    \end{enumerate}
\end{theorem}

\begin{proof}
    By meromorphicity, we only need to check both (1) and (2) within the convergence domain of the Eisenstein series.

    For $\varphi \in \cA_{P,\opn{cusp}}(G)$ and $\beta \in \opn{PW}(i\fa_{P}^*)$ (see ~\eqref{eq:Paley-Wiener}). Let
    \begin{equation*}
        \varphi'(g) = \int_{i\fa_P^*} e^{\langle \lambda, H_P(g) \rangle} \beta(\lambda) \varphi(g) \rd \lambda.
    \end{equation*}
    Then $\varphi' \in \cS_{\opn{cusp}}([G]_{P})$. Moreover, for $\lambda_0 \in \fa_{P}^*$ sufficiently positive, we have
    \begin{equation*}
        E(\varphi',g) = \int_{\lambda_0 + i \fa_{P}^* } E(g,\varphi,\lambda) \beta(\lambda) \rd \lambda.
    \end{equation*}

    Note that $E(h,\varphi, \lambda)$ forms a bounded set in $\cT([G])$ as $\lambda$ varies in $\lambda_0 + i \fa_{P}^*$. By theorem ~\ref{thm:Zydor_truncation_property} and lemma ~\ref{lem:intertwining_convergence}, both the integrals
    \[
        \int_{[H]} \int_{\lambda_0 + i\fa_{P}^*} \lvert E(h,\varphi,\lambda) \beta(\lambda) \rvert \rd \lambda \rd h.
    \]
    and
    \begin{equation*}
        \int_{\lambda_0 + i\fa_{P}^*} \int_{[H]_{P'}} \lvert \varphi(h) \beta(\lambda) \rvert e^{\langle \lambda_0 , H_P(g) \rangle} \rd h \rd \lambda
    \end{equation*}
    are convergent. By proposition ~\ref{prop:pseudo}, we then see that
    \begin{equation*}
        \int_{\lambda_0 + i\fa_P^*} \int_{[H]} E(h,\varphi,\lambda) \beta(\lambda) \rd \lambda = J(\varphi',\lambda) = \int_{\lambda_0 + i\fa_P^*} J(\varphi,\lambda) \beta(\lambda) \rd \lambda.
    \end{equation*}
    Since $\beta$ is arbitrary, (2) holds. The proof of (1) is similar (and indeed easier).
\end{proof}

\subsection{Proof of Theorem \ref{thm:intro_main}} 
\label{ssec:proof_of_intro_main}

By lemma ~\ref{eq:J_and_Whittaker}, let $\Pi= \Pi_n \otimes \Pi_{n+1}$ be a cuspidal automorphic representation of $\GL_n(\bA) \times \GL_{n+1}(\bA)$ with central character trivial on $A_{\GL_n}^\infty \times A_{\GL_{n+1}}^\infty$, if $\varphi \in \cA_{P,\Pi}$, by theorem \ref{thm:main} and \eqref{eq:J_and_Whittaker}, we have
\begin{equation} \label{eq:proof_main_1}
    \cP_H(E(\cdot,\varphi, \lambda)) = \int_{N_{2n+1}'(\bA) \backslash H(\bA)} W_{\varphi}(h) e^{\langle \lambda, H_{P}(h) \rangle}  \rd h,
\end{equation}
 for $\lambda \in \fa_{P,\C}^{*}$ such that $\mathrm{Re}(s_\lambda) \gg 0$.

Fix a decomposition $\Pi=\otimes^\prime_v \Pi_v$, $\psi=\otimes_v \psi_v$ and assume $\varphi=\otimes'_v \varphi_v$ is a pure tensor. Here each $\varphi_v$ lies in the local parabolic induction $\Ind_{P(F_v)}^{G(F_v)}\Pi_v$. Note the additive character $\psi_{n, n+1}$ we defined before can be decomposed as $\prod_v \psi_{n, n+1, v}$, where each $\psi_{n, n+1, v}$ is the additive character on $N_n(F_v) \times N_{n+1}(F_v)$.  Let $\cW(\Pi_v, \psi_v)$ denote the Whittaker model of the generic representation $\varphi_v$ with respect to $\psi_{n, n+1, v}$. We regard the parabolic induction $\Ind_{P(F_v)}^{G(F_v)}\cW(\Pi_v, \psi_v)$ as the spaces of functions $f: G(F_v) \to \mathbb{C}$ such that for any $g \in G(F_v)$,
\[  
    \delta_{P}^{-\frac{1}{2}}(m)f(mg), \quad m \in M(F_v)=\GL_n(F_v) \times \GL_{n+1}(F_v)
\]
lies in $\cW(\Pi_v, \psi_v)$ as a function on $\GL_n(F_v) \times \GL_{n+1}(F_v)$. Due to the uniqueness of the Whittaker model, we have
\begin{equation} \label{eq:proof_main_2}
    W_\varphi(g)= \prod_v W_{\varphi_v}(g_v), \quad g=(g_v)
\end{equation}
where $W_{\varphi_v}$ denotes the vector in $\Ind_{P(F_v)}^{G(F_v)}\cW(\Pi_v, \psi_v)$ associated to $\varphi_v$. For any local place $v$ and $\varphi_v \in \Ind_{P(F_v)}^{G(F_v)}\Pi_v$, we define the local period integral
\[
    \cP_{H,v}(\varphi_v, \lambda)= \int_{N_{2n+1}^\prime(F_v) \backslash H(F_v)}W_{\varphi_v}(h_v)e^{\langle \lambda, H_{P}(h_v) \rangle}  \rd h_v.
\]
By \eqref{eq:proof_main_1} and \eqref{eq:proof_main_2}, we have
\[
    \cP_H(E(\cdot, \varphi, \lambda))=(\Delta_{H}^*)^{-1} \prod_v \cP_{H,v}(\varphi_v, \lambda).
\]
Unfold the local period integral, and we can see
\begin{equation*}
    \begin{split}
        \cP_{H,v}(\varphi_v, \lambda)=\int_{P^\prime(F_v) \backslash H(F_v)} \int_{N_n(F_v) \backslash \GL_n(F_v)}&W_{R(h_v)\varphi_v}\begin{pmatrix}
        J{}^tm_v^{-1}J &  &  \\ &  m_v & \\ & & 1
    \end{pmatrix} \\ &\vert \det m_v \rvert ^{n+1+s_\lambda} e^{\langle \lambda, H_{P}(h_v) \rangle} \rd m_v \rd h_v.
    \end{split}
\end{equation*}
This integral converges absolutely for $\lambda$ such that $\mathrm{Re}(s_\lambda) \gg 0$, and extends meromorphically to all $\lambda \in \fa_{P,\C}^{*}$.

Let $W(\varphi_v)$ denote the vector in $\cW(\Pi_v, \psi_v)$ defined by
\[
    W(\varphi_v)(m)=\delta_{P}^{-\frac{1}{2}}(m)W_{\varphi_v}(m), \quad m \in \GL_n(F_v) \times \GL_{n+1}(F_v).
\]
Then we have
\[
    W_{\varphi_v}\begin{pmatrix}
        J{}^tm_v^{-1}J &  &  \\ &  m_v & \\ & & 1
    \end{pmatrix}= W(\varphi_v)(J{}^tm_v^{-1}J, \begin{pmatrix}
m_v &  \\
 & 1
\end{pmatrix})|\det m_v|^{-\frac{2n+1}{2}}
\]
and we can write
\begin{equation*}
    \begin{split}
        \cP_{H,v}(\varphi_v, \lambda) =\int_{P^\prime(F_v) \backslash H(F_v)} \int_{N_n(F_v) \backslash \GL_n(F_v)} & W(R(h_v)\varphi_v) \left(J{}^tm_v^{-1}J, \begin{pmatrix}
m_v &  \\
 & 1
\end{pmatrix} \right) \\ & \lvert \det m_v \rvert^{\frac{1}{2}+s_\lambda} e^{\langle \lambda, H_{P}(h_v) \rangle} \rd m_v \rd h_v.
    \end{split}
\end{equation*}
 Using Iwasawa decomposition, we have
\[
    \cP_{H,v}(\varphi_v, \lambda) = \int_{K_{H,v}}\int_{N_n(F_v) \backslash \GL_n(F_v)}W(R(k_v)\varphi_v) \left(J{}^tm_v^{-1}J \begin{pmatrix}
m_v &  \\
 & 1
\end{pmatrix} \right)\lvert \det m_v \rvert^{\frac{1}{2}+s_\lambda}\rd m_v \rd k_v
\]
for a suitable Haar measure on $K_{H,v}$. Moreover, when $\psi_v$ is unramified, $\opn{vol}(K_{H,v})=1$.
Assume $\Pi_v$ is unramified and let $\varphi_v^\circ$ be the unique spherical vector in $\Ind_{P(F_v)}^{G(F_v)}\Pi_v$ such that $W_{\varphi_v^\circ}(1)=1$. Then, by the unramified computation of the local Rankin-Selberg integral, we have
\[
    \cP_{H,v}(\varphi_v^\circ, \lambda)=L(s_\lambda+ 1, \Pi^\vee_{n,v} \times \Pi_{n+1, v}).
\]
If we define the local normalized period integral by
\[
    \cP_{H,v}^\natural(\varphi_v, \lambda)=\frac{\cP_{H,v}(\varphi_v, \lambda)}{L(s_\lambda+1, \Pi^\vee_{n,v} \times \Pi_{n+1, v})},
\]
then for any $\lambda \in \fa_{P,\C}^*$, we have
\[
    \cP_H(E(\cdot, \varphi, \lambda))=(\Delta_{H}^*)^{-1} L(s_\lambda+1, \Pi_n^\vee \times \Pi_{n+1})\prod_{v}\cP_{H,v}^\natural(\varphi_v, \lambda).
\]
In particular, setting $\lambda=0$ gives
\begin{equation} \label{eq:proof_main_3}
    \cP_H(E(\cdot, \varphi))=(\Delta_{H}^*)^{-1}L(1, \Pi_n^\vee \times \Pi_{n+1})\prod_{v}\cP_{H,v}^\natural(\varphi_v).
\end{equation}
Here we put $\cP_{H,v}^\natural(\varphi_v) :=\cP_{H,v}^\natural(\varphi_v, 0)$. Since $\Pi_n$ and $\Pi_{n+1}$ are both unitary, $\overline{L(1, \Pi_n^\vee \times \Pi_{n+1})}=L(1, \Pi_n \times \Pi_{n+1}^\vee)$. Taking conjugate of \eqref{eq:proof_main_3}, we have
\[
    \overline{\cP_H(E(\cdot, \varphi))}=(\Delta_{H}^*)^{-1}L(1, \Pi_n \times \Pi_{n+1}^\vee)\prod_{v} \overline{\cP_{H,v}^\natural(\varphi_v)},
\]
therefore
\begin{equation}\label{eq:sqperiod}
    |\cP_H(E(\cdot, \varphi))|^2=(\Delta_{H}^*)^{-2}L(1, \Pi_n^\vee \times \Pi_{n+1})L(1, \Pi_n \times \Pi_{n+1}^\vee)\prod_{v }|\cP_{H,v}^\natural(\varphi_v)|^2.
\end{equation}
We define the Petterson norm of $\varphi$ by
\[
    \langle \varphi, \varphi \rangle_{\mathrm{Pet}}= \int_{[G]_{P,0}} \lvert \varphi(x) \rvert^2 \rd x = \int_{P(\bA) \backslash G(\bA)} \int_{[M]_0}|\varphi(mg)|\delta_{P}(m)^{-1} \rd m \rd g.
\]
Let $\cP_n$ and $\cP_{n+1}$ denote the mirabolic subgroups of $\GL_n$ and $\GL_{n+1}$ respectively. We define the following local inner product
\[
    \langle \varphi_v, \varphi_v \rangle= \int_{P(F_v) \backslash G(F_v)} \int_{N_n(F_v)\times N_{n+1}(F_v) \backslash \cP_n(F_v) \times \cP_{n+1}(F_v)}|W(R(g_v)\varphi_v)(m_{1,v}, m_{2,v})|^2 \rd^* m_{1,v} \rd^* m_{2,v} \rd g_v.
\]
And we consider the normalization
\[
    \langle \varphi_v , \varphi_v \rangle^\natural=\frac{\Delta_{\GL_n,v} \Delta_{\GL_{n+1,v}} \langle \varphi_v, \varphi_v \rangle}{L(1, \Pi_{n,v}, \Ad)L(1, \Pi_{n+1, v}, \Ad)}.
\]
Then according to \cite[\S 4]{JacquetShalika}, we have
\begin{equation}\label{eq:norm}
    \langle \varphi, \varphi \rangle_{\mathrm{Pet}}=(\Delta_{G}^*)^{-1} L^*(1, \Pi_n, \Ad) L^*(1, \Pi_{n+1}, \Ad)\prod_{v}  \langle \varphi_v , \varphi_v \rangle^\natural.
\end{equation}
Note that
\[
    L^*(1, \Pi_n \boxplus \Pi_{n+1}, \Ad)=L^*(1, \Pi_n, \Ad) L^*(1, \Pi_{n+1}, \Ad)L(1, \Pi_n^\vee \times \Pi_{n+1})L(1, \Pi_n \times \Pi_{n+1}^\vee),
\]
and
\[
    L^*(1, \Pi, \Ad)=L^*(1, \Pi_n, \Ad) L^*(1, \Pi_{n+1}, \Ad).
\]
We devide \eqref{eq:sqperiod} by \eqref{eq:norm}, then we can write
\[
    \frac{|\cP_H(E(\cdot, \varphi))|^2}{\langle \varphi, \varphi \rangle_{\mathrm{Pet}}}=\frac{\Delta_{G}^*}{(\Delta_{H}^*)^2}\frac{L^*(1, \Pi_n \boxplus \Pi_{n+1}, \Ad)}{L^*(1, \Pi, \Ad)^2}\prod_{v}\frac{|\cP_{H,v}^\natural(\varphi_v)|^2}{\langle \varphi_v , \varphi_v \rangle^\natural}.
\]
Note the BZSV dual of $(\GL_{2n+1}, \Sp_{2n}, 0, 1)$ should be $(\GL_{2n+1}, \GL_n \times \GL_{n+1}, 0, 1)$. This implies that $\hat{\rho}_0$ is the adjoint representation of $\GL_n \times \GL_{n+1}$ on $\gl_{2n+1}$. So we can write $L^*(1, \Pi, \hat{\rho}_0)$ for $L^*(1, \Pi_n \boxplus \Pi_{n+1}, \Ad)$, and then
\[
    \frac{|\cP_H(E(\cdot, \varphi))|^2}{\langle \varphi, \varphi \rangle_{\mathrm{Pet}}}=\frac{\Delta_{G}^*}{(\Delta_{H}^*)^2}\frac{L^*(1, \Pi, \hat{\rho}_0)}{L^*(1, \Pi, \Ad)^2}\prod_{v}\frac{|\cP_{H,v}^\natural(\varphi_v)|^2}{\langle \varphi_v , \varphi_v \rangle^\natural}.
\]

\appendix

\section{Geometric conjecture (by Zeyu Wang)}
\label{sec:geometric}

In the appendix, we introduce the geometric analogues of various symplectic periods considered in this paper and formulate their cuspidal vanishing results.

\subsection{Conventions}

We work with a connected smooth projective curve $C$ over base field $K=\overline{\mathbb{F}}_{q}$. We use to denote the constructible etale sheaf theory with coefficient $k=\overline{\mathbb{Q}}_l$ for any prime number $l\neq \mathrm{char}(K)$ as developed in \cite{liu2017enhancedadicformalismperverse}. For each (higher) Artin stack $X$, we consider $\Shv(X)$ which is the category of ind-constructible \'etale $\overline{\mathbb{Q}}_l$-complexes over $X$. We have $\Shv(\mathrm{pt})=\Vect$ which is the category of complexes of vector spaces over $k$. For the purpose of this article, one can either regard the categories as triangulated categories or $\infty$-categories. 

The theory $\Shv$ enjoys a full six-functor formalism which we briefly recall as follows. For any map $f:X\to Y$ between Artin stacks, one has associated adjoint pair of functors $(f^*,f_*)$ where $f^*:\Shv(Y)\to \Shv(X)$ and $f_*:\Shv(X)\to \Shv(Y)$. When the map $f$ is locally of finite type, one has adjoint pair $(f_!,f^!)$ where $f_!:\Shv(X)\to\Shv(Y)$ and $f^!:\Shv(Y)\to \Shv(X)$. One also has an adjoint pair $(\otimes, \underline{\mathrm{Hom}})$. These functors enjoy many compatibilities and we refer to \cite{liu2017enhancedadicformalismperverse} for a complete list. We set $\Gamma_c=f_!$ and $\Gamma=f_*$ for $f:X\to \mathrm{pt}$ where $X$ is any Artin stack.

We use $\AS\in\Shv(\bA^1)$ to denote the Artin-Schreier sheaf (which depends on a choice of additive character $\psi:\mathbb{F}_q^{\times} \to k$ and we always fix such a choice). Denote $\AS_{-}:=[-1]^*\AS$ where $[-1]:\A^1\to \A^1$ is the map $[-1](x)=-x$.

For any algebraic group $G$ over $C$, we use $\Bun_G$ to denote the moduli stack of $G$-torsors over $C$.
We abbreviate $\Bun_{n}=\Bun_{\GL_n}$. For any subgroup $P\sub \GL_{2n}$, we denote $P':=P\cap \Sp_{2n}$. 

\subsection{Geometric symplectic periods}

Consider $\pi_{2n+1}:\Bun_{\Sp_{2n}}\to \Bun_{2n+1}$ induced by the natural inclusion $\Sp_{2n}\sub \GL_{2n+1}$, we define \[\int_{\Sp_{2n}}:\Shv(\Bun_{2n+1})\to\Vect\] via \[\int_{\Sp_{2n}}\cF:=\Gamma_c(\pi_{2n+1}^*\cF)\in\Vect, \cF\in\Shv(\Bun_{2n+1}).\]

\begin{theorem}\label{thm:geocuspvan}
    For any $\cF\in\Shv(\Bun_{2n+1})_{\cusp}$, we have $\int_{\Sp_{2n}}\cF=0$.
\end{theorem}

\begin{remark}
    The relation between theorem \ref{thm:geocuspvan} and \cite[Conjecture\,12.1.1]{BZSV} can be seen as follows: After restricting to $\Shv_{\mathrm{Nilp}}(\Bun_{2n+1})$ which is the category consisting of sheaves with nilpotent singular support, by \cite[Theorem\,3.4.6]{AGKRRV2} we have natural isomorphisms of functors \[\int_{\Sp_{2n}}\cong \Gamma_c(-\otimes\cP_X)\cong\ev(-\otimes\cP_X^r).\] Here $\cP_X$ is the period sheaf for $X=\GL_{2n+1}/\Sp_{2n}$ introduced in \cite[\S10.3]{BZSV}. The map $\ev:\Shv_{\mathrm{Nilp}}(\Bun_{2n+1})^{\otimes 2}\to\Vect$ is the counit for the miraculous duality on $\Shv_{\mathrm{Nilp}}(\Bun_{2n+1})$. The object $\cP_X^r\in\Shv_{\mathrm{Nilp}}(\Bun_{2n+1})$ is the \emph{right} spectral projection of $\cP_X$ introduced in \cite[\S12.4]{BZSV}. Therefore, theorem \ref{thm:geocuspvan} restricted to $\Shv_{\mathrm{Nilp}}(\Bun_{2n+1})$ is a consequence of \cite[Conjecture\,12.1.1]{BZSV} since the corresponding $L$-sheaf is supported on the reducible locus of $\mathrm{Loc}_{\check{G}}$. In particular, theorem \ref{thm:geocuspvan} verifies \cite[Conjecture\,12.1.1]{BZSV} for $X=\GL_{2n+1}/\Sp_{2n}$ on the cuspidal part (i.e. irreducible locus).
\end{remark}

Consider also $\pi_{2n}:\Bun_{\Sp_{2n}}\to\Bun_{2n}$ and $\pi_{\cL,2n}:\Bun_{\Sp_{2n}}^{\bA^{2n}_{\cL}}$ for any $\cL\in\Bun_1(K)$, where the later is the moduli of $(\cE_{2n}^s,\alpha:\cL\to\cE_{2n}^s)$ in which $\cE_{2n}^s\in\Bun_{\Sp_{2n}}$. Define also 
\[\int_{\Sp_{2n}}:\Shv(\Bun_{2n})\to\Vect\] via \[\int_{\Sp_{2n}}\cF:=\Gamma_c(\pi_{2n}^*\cF)\] and \[\int_{\bA^{2n}_{\cL}\times^{\Sp_{2n}}\GL_{2n}}:\Shv(\Bun_{2n})\to\Vect\] such that \[\int_{\bA^{2n}_{\cL}\times^{\Sp_{2n}}\GL_{2n}}\cF:=\Gamma_c(\pi_{\cL,2n}^*\cF).\]
We have 
\begin{theorem}\label{thm:geocuspvan2}
    For any $\cF\in\Shv(\Bun_{2n})_{\cusp}$, we have $\int_{\Sp_{2n}}\cF=0$.
\end{theorem}
\begin{theorem}\label{thm:geocuspvan3}
    For any $\cF\in\Shv(\Bun_{2n})_{\cusp}$, we have $\int_{\bA^{2n}_{\cL}\times^{\Sp_{2n}}\GL_{2n}}\cF=0$.
\end{theorem}

\subsection{Geometric Klingen-mirabolic period}
Now we introduce a geometric analogue of Klingen-mirabolic period and formulate its cuspidal vanishing result. This result will imply theorems \ref{thm:geocuspvan}, \ref{thm:geocuspvan2}, and \ref{thm:geocuspvan3}.
Consider the correspondence
\begin{equation}
\begin{tikzcd}
    \Bun_{P_{1,2n,1}'}\ar[r, "\pi_{1,2n,1}"] \ar[d, "b_{1,2n,1}"'] & \Bun_{P_{1,2n,1}} \\
    \Bun_{1}
\end{tikzcd}
\end{equation}
where we recall that \[\Bun_{P_{1,2n,1}'}= \{(\cE_1\sub\cE_{2n}^s)\}\] and the maps are defined by \[b_{1,2n,1}(\cE_{1}\sub\cE_{2n}^s)=\cE_1\in\Bun_1.\] \[\pi_{1,2n,1}(\cE_1\sub \cE_{2n}^s)=(\cE_1\sub\cE_1^{\perp}\sub\cE_{2n}^s)\in\Bun_{P_{1,2n,1}}.\] Here we use $\cE_k$ (later we will also use $\cE_k'$) to denote an arbitrary choice of vector bundle of rank $n$. We use $\cE_{k}^s$ to denote an arbitrary choice of a symplectic vector bundle of rank $n$. We use $\cE_a\sub\cE_b$ to denote an arbitrary choice of bundle inclusion (i.e. $\cE_a$ should be saturated in $\cE_b$).

We define $\int_{1,2n,1}:\Shv(\Bun_{P_{1,2n,1}})\to\Shv(\Bun_{1})$ by \[\int_{1,2n,1}\cF:=b_{1,2n,1,!}\pi_{1,2n,1}^*\in\Shv(\Bun_1), \cF\in\Shv(\Bun_{P_{1,2n,1}}).\]

For any $a,b\in\mathbb{Z}_{\geq 0}$ such that $2a+b=2n$, consider the correspondence
\begin{equation}
    \begin{tikzcd}
        \Bun_{P_{1,a,b,a,1}} \ar[r, "p_{a,b,a}"] \ar[d, "q_{a,b,a}"'] & \Bun_{P_{1,2n,1}} \\
        \Bun_{P_{1,a}}\times\Bun_b\times \Bun_{P_{a,1}}.
    \end{tikzcd}
\end{equation}
We define the functor $\CT_{a,b,a}:\Shv(\Bun_{P_{1,2n,1}})\to\Shv(\Bun_{P_{1,a}}\times\Bun_b\times \Bun_{P_{a,1}})$ by \[\CT_{a,b,a}(\cF):=q_{a,b,a,!}p_{a,b,a}^*\cF.\] 

We define $\Shv(\Bun_{P_{1,2n,1}})_{\theta-\cusp}\sub \Shv(\Bun_{P_{1,2n,1}})$ to be the full-subcategory such that $\cF\in\Shv(\Bun_{P_{1,2n,1}})_{\theta-\cusp}$ if \[\CT_{a,b,a}(\cF)=0\] for all $a,b$ as above.

We have the following vanishing result for the geometric Klingen-mirabolic period:
\begin{lemma}\label{lem:geopcuspvan}
    For any $n\in\mathbb{Z}_{\geq 0}$ and $\cF\in\Shv(\Bun_{P_{1,2n,1}})_{\theta-\cusp}$, we have $\int_{1,2n,1}\cF=0$.
\end{lemma}

\subsection{Proofs}

\begin{proof}[Proof of \ref{thm:geocuspvan} assuming lemma \ref{lem:geopcuspvan}] In this section, we give proofs to the results stated before. The proofs are mostly direct translations of their classical counterparts.

    Consider the diagram
    \begin{equation}
        \begin{tikzcd}
            S_{P_{2n,1}}^{\circ} \ar[r, hook, "j_{2n,1}"] \ar[rr, bend left, "q_{2n,1}^{\vee,\circ}"] & S_{P_{2n,1}} \ar[r, "q_{2n,1}^\vee"] & \Bun_{2n}\times \Bun_1 \ar[r, bend left, "s_{2n,1}"] \ar[l, bend left, "s_{2n,1}^{\vee}"] & \Bun_{P_{2n,1}} \ar[l, "q_{2n,1}"'] \ar[r, "p_{2n,1}"] & \Bun_{2n+1} \\
            &&\Bun_{\Sp_{2n}} \ar[u, "\pi_{2n,1}"] &&
        \end{tikzcd}
    \end{equation}
Here, the maps are defined as 
\[\pi_{2n,1}(\cE_{2n}^s)=(\cE_{2n}^s,\cO)\]
\[q_{2n,1}(\cE_{2n}\sub\cE_{2n+1})=(\cE_{2n},\cE_{2n+1}/\cE_{2n})\]
\[p_{2n,1}(\cE_{2n}\sub\cE_{2n+1})=\cE_{2n+1}\]
and $q_{2n,1}^\vee:S_{P_{2n,1}}\to \Bun_{2n}\times\Bun_1$ is the underlying classical stack of the (derived) dual vector bundle of $q_{2n,1}:\Bun_{P_{2n,1}}\to\Bun_{2n}\times\Bun_1$. The maps $s_{2n,1}$ and $s_{2n,1}^{\vee}$ are the zero-section maps.

Note that \[S_{P_{2n,1}}=\{(\cE_1,\cE_{2n},\alpha_{2n}:\cE_1\otimes\om^{-1}\to\cE_{2n})\}.\] The open substack $S^{\circ}_{P_{2n,1}}\sub S_{P_{2n,1}}$ is cut out by the condition such that $\alpha_{2n}\neq 0$.

We use $\FT_{2n,1}:\Shv(\Bun_{P_{2n,1}})\to\Shv(S_{P_{2n,1}})$ to denote the (derived) Fourier transformation functor introduced in \cite[\S6.1]{FYZ3}. Now we recall its definition as follows. Consider the following diagram \begin{equation}
\begin{tikzcd}
    \bA^1 & S_{P_{2n,1}}\times_{\Bun_{2n}\times\Bun_1} \Bun_{P_{2n,1}} \ar[r, "\widetilde{q}_{2n,1}^{\vee}"] \ar[d, "\widetilde{q}_{2n,1}"] \ar[l, "\ev_{2n,1}"']  & \Bun_{P_{2n,1}} \ar[d, "q_{2n,1}"] \\
    &S_{P_{2n,1}} \ar[r, "q_{2n,1}^{\vee}"] & \Bun_{2n}\times\Bun_1 \arrow[ul, phantom, "\lrcorner", very near end]
    \end{tikzcd}
\end{equation} where the map $\ev_{2n,1}$ is the natural evaluation map between dual vector bundles. We have the (derived) Fourier transformation functor \[\FT_{2n,1}:\Shv(\Bun_{P_{2n,1}})\to\Shv(S_{P_{2n,1}})\] defined by \[\FT_{2n,1}(\cF)=\tilq_{2n,1,!}(\tilq_{2n,1}^{\vee,*}\cF\otimes \ev_{2n,1}^*\AS[d_{2n,1}]).\] The inverse of this functor is given by \[\FT_{2n,1}^{-1}:\Shv(S_{P_{2n,1}})\to \Shv(\Bun_{P_{2n,1}})\] given by \[\FT_{2n,1}^{-1}(\cG)=\tilq_{2n,1,!}^{\vee}(\tilq_{2n,1}^*\cG\otimes\ev_{2n,1}^*\AS_{-}[d_{2n,1}]).\] Here we use $d_{2n,1}$ to denote the relative (virtual) dimension of the vector bundle $q_{2n,1}:\Bun_{P_{2n,1}}\to\Bun_{2n}\times \Bun_1$ (which varies on different connected components of $\Bun_{2n}\times\Bun_1$).
\begin{lemma}\label{lem:geounfold1small}
    For any $\cF\in\Shv(\Bun_{2n+1})_{\cusp}$, we have $s_{2n,1}^{\vee,*}\FT_{2n,1}(p_{2n,1}^*\cF)=0$.
\end{lemma}
\begin{proof}
    This follows directly as \[s_{2n,1}^{\vee,*}\FT_{2n,1}(p_{2n,1}^*\cF)=q_{2n,1,!}p_{2n,1}^*\cF[d_{2n,1}]=\CT_{2n,1}(\cF)[d_{2n,1}]=0.\]
\end{proof}

Consider the stack $M_{2n,1}:=\Bun_{\Sp_{2n}}\times_{\Bun_{2n}\times\Bun_1} S_{P_{2n,1}}^{\circ}\times_{\Bun_{2n}\times\Bun_1}\Bun_{P_{2n,1}}$, which has moduli description \[M_{2n,1}=\{(\cE_{2n}^s\sub \cE_{2n+1}, \cE_{2n+1}/\cE_{2n}^s\cong \cO, \om^{-1}\inj \cE_{2n}^s) \}\]

There is a natural map \[f_{2n,1}:=(p_{2n,1}\circ\pr_3)\times (\ev_{2n,1}\circ (\pi_{2n}\times \id )\circ \pr_{1,2}): M_{2n,1}\to \Bun_{2n+1}\times \bA^1.\]

We have \begin{equation}
\begin{split}
    \int_{\Sp_{2n}}\cF&=\Gamma_c(\pi_{2n,1}^*s_{2n,1}^*p_{2n,1}^*\cF) \\
    &\cong \Gamma_c(\pi_{2n,1}^*s_{2n,1}^*\FT_{2n,1}^{-1}\circ \FT_{2n,1}( p_{2n,1}^*\cF)) \\
    &\cong \Gamma_c(\pi_{2n,1}^*q_{2n,1,!}^{\vee} \FT_{2n,1}( p_{2n,1}^*\cF))[d_{2n,1}] \\
    &\cong \Gamma_c(\pi_{2n,1}^*q_{2n,1,!}^{\vee,\circ}j_{2n,1}^*\FT_{2n,1}( p_{2n,1}^*\cF))[d_{2n,1}] \\
    &\cong \Gamma_c(f_{2n,1}^*(\cF\boxtimes \AS))[2d_{2n,1}].
    \end{split}
\end{equation}

Consider the moduli stack \[\tilM_{2n,1}=\{(\cE_1\sub\cE_{2n}^s\sub\cE_{2n+1},\cE_{2n+1}/\cE_{2n}^s\cong\cO,\om^{-1}\inj \cE_1)\}.\] There is a natural map $\iota_{2n,1}:\tilM_{2n,1}\to M_{2n,1}$ given by \[\iota(\cE_1\sub\cE_{2n}^s\sub\cE_{2n+1},\cE_{2n+1}/\cE_{2n}^s\cong\cO,\om^{-1}\inj \cE_1)=(\cE_{2n}^s\sub \cE_{2n+1}, \cE_{2n+1}/\cE_{2n}^s\cong \cO, \om^{-1}\inj \cE_1\sub \cE_{2n}^s).\] Note that the map $\iota_{2n,1}$ is a bijection on points and induces a stratification of $M_{2n,1}$. Denote \[\tilf_{2n,1}:=f_{2n,1}\circ \iota_{2n,1}:\tilM_{2n,1}\to \Bun_{2n+1}\times\bA^1.\] We are reduced to show \begin{equation}
    \Gamma_c(\tilf_{2n,1}^*(\cF\boxtimes \AS))=0.
\end{equation}
Consider the moduli stack \[\tilN_{2n,1}=\{(\cE_1\sub\cE_{2n-1}\sub \cE_{2n}\sub\cE_{2n+1},\cE_{2n+1}/\cE_{2n}\cong\cO,\cE_{2n}/\cE_{2n-1}\inj \om)\}.\] We have a diagram
\begin{equation}\label{diag:geofund1}
\begin{tikzcd}
    \tilM_{2n,1} \ar[rr, bend left, "\tilf_{2n,1}"] \ar[r, "\pi_{M,2n,1}"] \ar[d, "q_{M,2n,1}"]  & \tilN_{2n,1} \ar[r, "f_{N,2n,1}"] \ar[d, "q_{N,2n,1}"] & \Bun_{2n+1}\times \bA^1 \\
    \Bun_{P_{1,2n-2,1}'} \ar[r, "\pi_{1,2n-2,1}"] & \Bun_{P_{1,2n-2,1}} \arrow[ul, phantom, "\lrcorner", very near end] & 
    \end{tikzcd}
\end{equation}
where the maps are given by
\[\begin{split}\pi_{M,2n,1}(\cE_1\sub\cE_{2n}^s\sub\cE_{2n+1},\cE_{2n+1}/\cE_{2n}^s\cong\cO,\om^{-1}\inj\cE_1)\\=(\cE_1\sub\cE_1^{\perp}\sub \cE_{2n}^s\sub\cE_{2n+1},\cE_{2n+1}/\cE_{2n}^s\cong\cO,\cE_{2n}^s/\cE_{2n-1}\cong \cE_1^*\inj \om)\end{split}\]
\[q_{M,2n,1}(\cE_1\sub\cE_{2n}^s\sub\cE_{2n+1},\cE_{2n+1}/\cE_{2n}^s\cong\cO,\cE_{2n}^s/\cE_{2n-1}\inj \om)=(\cE_1\sub\cE_{2n}^s)\]
\[q_{N,2n,1}(\cE_1\sub\cE_{2n-1}\sub \cE_{2n}\sub\cE_{2n+1},\cE_{2n+1}/\cE_{2n}\cong\cO,\cE_{2n}/\cE_{2n-1}\inj \om)=(\cE_1\sub\cE_{2n-1}\sub\cE_{2n}).\] Note that the square in \eqref{diag:geofund1} is Cartesian, and $\tilf_{2n,1}=f_{N,2n,1}\circ\pi_{M,2n,1}$, we get \begin{equation}
    \Gamma_c(\tilf_{2n,1}^*(\cF\boxtimes \AS))=\Gamma_c(\pi_{1,2n-2,1}^*q_{N,2n,1,!}f_{N,2n,1}^*(\cF\boxtimes\AS)).
\end{equation} 
\begin{lemma}\label{lem:geoctoc1}
    Consider the diagram \begin{equation}
        \begin{tikzcd}
            \Bun_{P_{1,2n-2,1,1}} \ar[r, "p_{1,2n-2,1,1}\times  q_{1,1}'"] \ar[d, "q_{1,2n-2,1,1}"] & \Bun_{2n+1}\times \Bun_{P_{1,1}} \\
            \Bun_{P_{1,2n-2,1}}
        \end{tikzcd}
    \end{equation}
    where the maps are \[p_{1,2n-2,1,1}(\cE_1\sub\cE_{2n-1}\sub\cE_{2n}\sub\cE_{2n+1})=\cE_{2n+1}\] \[q_{1,1}'(\cE_1\sub\cE_{2n-1}\sub\cE_{2n}\sub\cE_{2n+1})=(\cE_{2n}/\cE_{2n-1}\sub\cE_{2n+1}/\cE_{2n-1})\] \[q_{1,2n-2,1,1}(\cE_1\sub\cE_{2n-1}\sub\cE_{2n}\sub\cE_{2n+1})=(\cE_1\sub\cE_{2n-1}\sub\cE_{2n}).\] Then for any $\cF\in\Shv(\Bun_{2n+1})_{\cusp}$, $\cG\in\Shv(\Bun_{P_{1,1}})$, we have \[q_{1,2n-2,1,1,!}(p_{1,2n-2,1,1}\times q_{1,1}')^*(\cF\boxtimes \cG)\in\Shv(\Bun_{P_{1,2n-2,1}})_{\theta-\cusp}.\]
\end{lemma}
\begin{proof}[Proof of lemma \ref{lem:geoctoc1}]
    For any $a,b\in\mathbb{Z}_{\geq 0}$ such that $2a+b=2n-2$, consider the diagram
    \begin{equation}
    \adjustbox{max width=\linewidth}{%
    \begin{tikzcd}[column sep=small]
        \Bun_{P_{1,2n-2,1}} \arrow[dr, phantom, "\ulcorner", very near end] & \Bun_{P_{1,2n-2,1,1}} \ar[r] \ar[l] & \Bun_{2n+1} \times \Bun_{P_{1,1}} \\
        \Bun_{P_{1,a,b,a,1}} \ar[u] \ar[d] & \Bun_{P_{1,a,b,a,1,1}} \ar[u] \ar[l] \ar[r]\ar[d] & \Bun_{P_{a+1,b,a+2}}\times\Bun_{P_{1,1}} \ar[u]\ar[d] \\
        \Bun_{P_{1,a}}\times \Bun_b\times \Bun_{P_{a,1}} & \Bun_{P_{1,a}}\times \Bun_b\times\Bun_{P_{a,1,1}}\ar[l]\ar[r]& \Bun_{a+1}\times\Bun_b\times\Bun_{a+2}\times\Bun_{P_{1,1}} \arrow[ul, phantom, "\lrcorner", very near end]
        \end{tikzcd}
        }
    \end{equation} where the maps are the obvious ones. Note that the functor \[\CT_{a,b,a}\circ q_{1,2n-2,1,1,!}(p_{1,2n-2,1,1}\times q_{1,1}')^*\] is given by transformation along the top and left correspondences. Since the left-upper and right-lower squares are Cartesian, the functor above is also given by the composition of transformations along the bottom and right correspondences. Note that the right correspondence transforms $\cF\boxtimes\cG$ to zero since $\cF\in\Shv(\Bun_{2n+1})_{\cusp}$, and we are done.
\end{proof}

Now we apply lemma \ref{lem:geoctoc1} to conclude the proof. By lemma \ref{lem:geopcuspvan}, we only need to show \begin{equation}\label{eq:geocusp1}q_{N,2n,1,!}f_{N,2n,1}^*(\cF\boxtimes\AS)\in\Shv(\Bun_{P_{1,2n,1}})_{\theta-\cusp}.\end{equation} Consider the diagram \begin{equation}
    \begin{tikzcd}
        \tilN_{1,1} \ar[r, "\ev_{1,1}"] \ar[d, "q_{N,1,1}"] & \bA^1 \\
        \Bun_{P_{1,1}}
    \end{tikzcd}
\end{equation}
where we define \[\tilN_{1,1}=\{(\cE_1\sub\cE_2,\cE_2/\cE_1\cong\cO,\cE_1\inj\om)\}\]\[q_{N,1,1}(\cE_1\sub\cE_2,\cE_2/\cE_1\cong\cO,\cE_1\inj\om)=(\cE_1\sub\cE_2).\] Then one easily sees that \[q_{N,2n,1,!}f_{N,2n,1}^*(\cF\boxtimes\AS)\cong q_{1,2n-2,1,1,!}(p_{1,1}\times q_{1,2n-2,1,1}')^*(\cF\boxtimes \cG)\] for $\cG=q_{N,1,1,!}\ev_{1,1}^*\AS$. Then the desired cuspidality \eqref{eq:geocusp1} follows from lemma \ref{lem:geoctoc1}.

\end{proof}

\begin{proof}[Proof of lemma \ref{lem:geopcuspvan}]
    The proof follows the same ideas as the proof of theorem \ref{thm:geocuspvan}.
    We proceed by induction on $n$. The case $n=0$ is trivial. From now on, we assume $n\geq 1$. Consider the quotient group $\overline{P}_{1,2n,1}:=P_{1,2n,1}/\Ga$ and $\overline{P}_{1,2n,1}':=P_{1,2n,1}'/\Ga$, where the subgroup $\Ga\sub P_{1,2n,1}$ identifies with \begin{equation*}
    \begin{pmatrix}
         &  & * \\
          &   &  \\
          &          & 
    \end{pmatrix}\sub
    \begin{pmatrix}
        * & * & * \\
          &  * & * \\
          &          & *
    \end{pmatrix}
\end{equation*} and similarly for $\Ga\sub P_{1,2n,1}'$. 
 
Consider diagram
\begin{equation}
\begin{tikzcd}
    \Bun_{P_{1,2n,1}'}\ar[r, "i_{2n}"] \ar[rr, bend left, "\pi_{1,2n,1}"] \ar[d, "r_{2n}'"] & \Bun_{P_{1,2n,1}}^s \ar[r, "h_{2n}"] \ar[d, "r_{2n}^s"] & \Bun_{P_{1,2n,1}} \ar[d, "r_{2n}"] \\
    \Bun_{\overline{P}_{1,2n,1}'} \ar[r, "\overline{i}_{2n}"] \ar[dr, "q_{1,2n,1}'"] & \Bun_{\overline{P}_{1,2n,1}}^s \ar[r, "\overline{h}_{2n}"] \ar[d, "q_{1,2n,1}^s"] \arrow[ul, phantom, "\lrcorner", very near end] & \Bun_{\overline{P}_{1,2n,1}} \ar[d, "q_{1,2n,1}"] \arrow[ul, phantom, "\lrcorner", very near end] \\
     \Bun_1 & \Bun_{\Sp_{2n}}\times \Bun_1\ar[r, "\pi_{2n}\times \Delta_{-} "] \ar[l, "\pr_2"'] & \Bun_{2n}\times \Bun_1\times\Bun_1 \arrow[ul, phantom, "\lrcorner", very near end] \arrow[dl, phantom, "\urcorner", very near end] \\
     S_{\overline{P}_{1,2n,1}'} \ar[ur, "q_{1,2n,1}'^{\vee}"] & S^{s}_{\overline{P}_{1,2n,1}} \ar[l, "\overline{i}_{2n}^{\vee}"] \ar[r, "\overline{h}_{2n}^{\vee}"] \ar[u, "q_{1,2n,1}^{s,\vee}"'] & S_{\overline{P}_{1,2n,1}} \ar[u, "q_{1,2n,1}^{\vee}"'] \arrow[dl, phantom, "\urcorner", very near end] \\
      & S^{s,\circ}_{\overline{P}_{1,2n,1}}  \ar[r, "\overline{h}_{2n}^{\vee,\circ}"] \ar[u, "j_{1,2n,1}^s"'] & S^{\circ}_{\overline{P}_{1,2n,1}} \ar[u, "j_{1,2n,1}"']
     \end{tikzcd}.
\end{equation} We now explain the definition of stacks involved in the diagram. The stacks in the forth row are defined as underlying classical stacks of the derived dual vector bundles of the stacks in the second row. The stacks in the bottom row are obtained from the vector bundles in the forth row the complements of zero sections. The stacks in the right most column are clear as written. The stacks in the middle column are defined to make the four squares on the right Cartesian. The stacks in the left most columns are also clear as written. The map $i_{2n}$ is defined such that $h_{2n}\circ i_{2n}=\pi_{1,2n,1}$. The map $\overline{i}_{2n}^{\vee}$ is the dual map of $\overline{i}_{2n}$. In terms of moduli stacks, we have \[\Bun_{\overline{P}'_{1,2n,1}}=\{(\cE_1\sub \cE_{2n+1},\cE_{2n}^s,\cE_{2n}^s\cong\cE_{2n+1}/\cE_{1})\}\] \[\Bun_{\overline{P}_{1,2n,1}}^s=\{(\cE_1\sub\cE_{2n+1},\cE_{2n}^s\sub\cE_{2n+1}',\cE_{2n+1}/\cE_{1}\cong \cE_{2n}^s, \cE_{1}^*\cong \cE_{2n+1}'/\cE_{2n}^s)\}\]
\[S_{\overline{P}_{1,2n,1}'}=\{(\cE_1,\cE_{2n}^s,\cE_1\otimes\om^{-1}\to\cE_{2n}^s)\}\]
\[S^s_{\overline{P}_{1,2n,1}}=\{(\cE_1,\cE_{2n}^s;\alpha_1,\alpha_2:\cE_1\otimes\om^{-1}\to\cE_{2n}^s)\}\]
 We have moduli description of the maps 
\[\overline{i}_{2n}(\cE_1\sub \cE_{2n+1},\cE_{2n}^s,\cE_{2n}^s\cong\cE_{2n+1}/\cE_{1})=(\cE_1\sub\cE_{2n+1},\cE_{2n}^s\sub\cE_{2n+1}^*,\cE_{2n+1}/\cE_{1}\cong \cE_{2n}^s, \cE_{1}^*\cong \cE_{2n+1}^*/\cE_{2n}^s)\]
\[\overline{i}_{2n}^{\vee}(\cE_1,\cE_{2n}^s;\alpha_1,\alpha_2:\cE_1\otimes\om^{-1}\to\cE_{2n}^s)=(\cE_1,\cE_{2n}^s,\alpha_1-\alpha_2:\cE_1\otimes\om^{-1}\to\cE_{2n}^s).\]
\[\Delta_{-}(\cE_1)=(\cE_1,\cE_1^*).\]

Consider the derived Fourier transformations \[\FT_{1,2n,1}:\Shv(\Bun_{\overline{P}_{1,2n,1}})\to\Shv(S_{\overline{P}_{1,2n,1}})\] \[\FT^s_{1,2n,1}:\Shv(\Bun^s_{\overline{P}_{1,2n,1}})\to\Shv(S^s_{\overline{P}_{1,2n,1}})\]
\[\FT'_{1,2n,1}:\Shv(\Bun_{\overline{P}'_{1,2n,1}})\to\Shv(S'_{\overline{P}_{1,2n,1}}).\]

\begin{lemma}
    For any $\cF\in\Shv(\Bun_{P_{1,2n,1}})_{\theta-\cusp}$, the sheaf $\FT_{1,2n,1}(r_{2n,!}\cF)\in\Shv(S_{\overline{P}_{1,2n,1}})$ is supported on the open substack $S_{\overline{P}_{1,2n,1}}^{\circ}\sub S_{\overline{P}_{1,2n,1}}$.
\end{lemma}
The proof is similar to the proof of lemma \ref{lem:geounfold1small}.

We use $d_{2n}$ to denote the (virtual) relative dimension of the vector bundle $q_{1,2n,1}:\Bun_{\overline{P}_{1,2n,1}}\to\Bun_{2n}\times\Bun_1\times\Bun_1$, and $d_{2n}'$ to denote the (virtual) relative dimension of $q_{1,2n,1}':\Bun_{\overline{P}'_{1,2n,1}}\to\Bun_{\Sp_{2n}}\times\Bun_1$.

Consider the stack \[M_{1,2n,1}=(\Bun_{\Sp_{2n}}\times\Bun_1)\times_{S_{\overline{P}_{1,2n,1}'}} S_{\overline{P}_{1,2n,1}}^{s,\circ}\times_{\Bun_{2n}\times\Bun_1\times\Bun_1} \Bun_{P_{1,2n,1}}.\] It has moduli description \[M_{1,2n,1}=\{(\cE_1\sub\cE_{2n+1}\sub\cE_{2n+2},\cE_{2n}^s,\cE_{2n}^s\cong\cE_{2n+1}/\cE_1,\cE_1^*\cong\cE_{2n+2}/\cE_{2n+1},\cE_{1}\otimes\om^{-1}\inj\cE_{2n}^s)\}.\]  It is equipped with canonical maps \[f_{1,2n,1}:=(\pr_3,\ev_{1,2n,1}\circ\pr_{2,3}):M_{1,2n,1}\to \Bun_{P_{1,2n,1}}\times \bA^1\] \[c_{1,2n,1}:=\pr_2\circ\pr_1:M_{1,2n,1}\to \Bun_1.\] Here $\ev_{1,2n,1}:S_{\overline{P}_{1,2n,1}}^{s,\circ}\times_{\Bun_{2n}\times\Bun_1\times\Bun_1} \Bun_{P_{1,2n,1}}\to \bA^1$ is given by the restriction of the evaluation map between dual vector bundles $q_{1,2n,1}:\Bun_{\overline{P}_{1,2n,1}}\to\Bun_{2n}\times\Bun_1\times\Bun_1$ and $q_{1,2n,1}^{\vee}:S_{\overline{P}_{1,2n,1}}\to\Bun_{2n}\times\Bun_1\times\Bun_1$.

Now we can compute for any $\cF\in\Shv(\Bun_{P_{1,2n,1}})_{\theta-\cusp}$:
\begin{equation}
\begin{split}
    \int_{1,2n,1}\cF&\cong\pr_{2,!}q_{1,2n,1,!}'r_{2n,!}'i_{2n}^*h_{2n}^*\cF\\
    &\cong\pr_{2,!}q_{1,2n,1,!}'\overline{i}_{2n}^*\overline{h}_{2n}^*r_{2n,!}\cF \\
    &\cong\pr_{2,!}q_{1,2n,1,!}'\overline{i}_{2n}^*\overline{h}_{2n}^*\FT_{1,2n,1}^{-1}\circ\FT_{1,2n,1}(r_{2n,!}\cF) \\
    &\cong\pr_{2,!}q_{1,2n,1,!}'\overline{i}_{2n}^*\FT_{1,2n,1}^{s,-1}(\overline{h}_{2n}^{\vee,*}\FT_{1,2n,1}(r_{2n,!}\cF)) \\
    &\cong\pr_{2,!}q_{1,2n,1,!}'\FT_{1,2n,1}'^{-1}(\overline{i}_{2n,!}^{\vee}\overline{h}_{2n}^{\vee,*}\FT_{1,2n,1}(r_{2n,!}\cF))[d_{2n}-d_{2n}']\\
    &\cong \pr_{2,!}s_{1,2n,1}'^{\vee,*}\overline{i}_{2n,!}^{\vee}\overline{h}_{2n}^{\vee,*}\FT_{1,2n,1}(r_{2n,!}\cF)[d_{2n}-2d_{2n}'] \\
    &\cong c_{1,2n,1,!}f_{1,2n,1}^*(\cF\boxtimes \AS)[2d_{2n}-2d'_{2n}]
    \end{split}
\end{equation} where the map $s_{1,2n,1}'^{\vee}:\Bun_{\Sp_{2n}}\times\Bun_1\to S_{\overline{P}_{1,2n,1}'}$ is the zero section of the vector bundle $q_{1,2n,1}'^{\vee}:S_{\overline{P}_{1,2n,1}'}\to \Bun_{\Sp_{2n}}\times\Bun_1$.

Consider the moduli stack \[\tilM_{1,2n,1}:=\{(\cE_1\sub\cE_{2n+1}\sub\cE_{2n+2},\cE_1'\sub \cE_{2n}^s,\cE_{2n}^s\cong\cE_{2n+1}/\cE_1,\cE_1^*\cong\cE_{2n+2}/\cE_{2n+1},\cE_{1}\otimes\om^{-1}\inj\cE_1')\}.\] It is equipped with a natural map \[\iota_{1,2n,1}:\tilM_{1,2n,1}\to M_{1,2n,1}\] \[\begin{split}\iota_{1,2n,1}(\cE_1\sub\cE_{2n+1}\sub\cE_{2n+2},\cE_1'\sub \cE_{2n}^s,\cE_{2n}^s\cong\cE_{2n+1}/\cE_1,\cE_1^*\cong\cE_{2n+2}/\cE_{2n+1},\cE_{1}\otimes\om^{-1}\inj\cE_1')\\=(\cE_1\sub\cE_{2n+1}\sub\cE_{2n+2}, \cE_{2n}^s,\cE_{2n}^s\cong\cE_{2n+1}/\cE_1,\cE_1^*\cong\cE_{2n+2}/\cE_{2n+1},\cE_{1}\otimes\om^{-1}\inj\cE_1'\sub\cE_{2n}^s) .\end{split}\] This map is a bijection on points and induces a stratification on $M_{1,2n,1}$. Consider the commutative diagram \begin{equation}
    \begin{tikzcd}
        &\tilM_{1,2n,1} \ar[dl, "\tilc_{1,2n,1}"'] \ar[d, "\iota_{1,2n,1}"] \ar[dr, "\tilf_{1,2n,1}"]& \\
        \Bun_1 & M_{1,2n,1} \ar[l, "c_{1,2n,1}"'] \ar[r, "f_{1,2n,1}"] & \Bun_{P_{1,2n,1}}\times\bA^1
    \end{tikzcd}.
\end{equation} We are reduced to prove \[\tilc_{1,2n,1,!}\tilf^*_{1,2n,1}(\cF\boxtimes \AS)=0\] for any $\cF\in\Shv(\Bun_{P_{1,2n,1}})_{\theta-\cusp}$.

Consider the moduli stack \[\tilN_{1,2n,1}=\{(\cE_1\sub\cE_2\sub\cE_{2n}\sub\cE_{2n+1}\sub\cE_{2n+2},\cE_1\otimes\om^{-1}\inj\cE_2/\cE_1)\}.\] Then we have a natural map $\pi_{M,1,2n,1}:\tilM_{1,2n,1}\to\tilN_{1,2n,1}$ \[\begin{split}\pi_{M,1,2n,1}(\cE_1\sub\cE_{2n+1}\sub\cE_{2n+2},\cE_1'\sub \cE_{2n}^s,\cE_{2n}^s\cong\cE_{2n+1}/\cE_1,\cE_1\cong\cE_{2n+2}/\cE_{2n+1},\cE_{1}\otimes\om^{-1}\inj\cE_1')\\=(\cE_1\sub\cE_2\sub\cE_{2n}'\sub\cE_{2n+1}\sub\cE_{2n+2},\cE_1\otimes\om^{-1}\inj\cE_1'\cong \cE_2/\cE_1)\end{split}\] where $\cE_2$, $\cE_{2n}'$ are the preimages of $\cE_1'$, $\cE_1'^{\perp}$ under $\cE_{2n+1}\surj\cE_{2n}^s$.

Consider the diagram in which the square is Cartesian \begin{equation}
    \begin{tikzcd}
        \tilM_{1,2n,1} \ar[r, "\pi_{M,1,2n,1}"] \ar[d, "q_{M,1,2n,1}"] & \tilN_{1,2n,1} \ar[d, "q_{N,1,2n,1}"] \ar[r, "f_{N,1,2n,1}"] & \Bun_{P_{1,2n,1}}\times\bA^1 \\
        \Bun_{P'_{1,2n-2,1}}\times\Bun_1 \ar[r, "\pi_{1,2n,1}\times \Delta_{-}"'] \ar[d, "\pr_2"] & \Bun_{P_{1,2n-2,1}}\times\Bun_1\times\Bun_1 \arrow[ul, phantom, "\lrcorner", very near end] & \\
        \Bun_1 &&
    \end{tikzcd}
\end{equation}
where the maps are defined as \[\begin{split}q_{M,1,2n,1}(\cE_1\sub\cE_{2n+1}\sub\cE_{2n+2},\cE_1'\sub \cE_{2n}^s,\cE_{2n}^s\cong\cE_{2n+1}/\cE_1,\cE_1^*\cong\cE_{2n+2}/\cE_{2n+1},\cE_{1}\otimes\om^{-1}\inj\cE_1')\\=(\cE_1'\sub\cE_1'^{\perp}\sub\cE_{2n}^s,\cE_1) \end{split}\]
\[\begin{split}q_{N,1,2n,1}(\cE_1\sub\cE_2\sub\cE_{2n}\sub\cE_{2n+1}\sub\cE_{2n+2},\cE_1\otimes\om^{-1}\inj\cE_2/\cE_1)\\=(\cE_2/\cE_1\sub\cE_{2n}/\cE_1\sub\cE_{2n+1}/\cE_1,\cE_1,\cE_{2n+2}/\cE_{2n+1})\end{split}\]
\[\Delta_{-}(\cE_1)=(\cE_1,\cE_1^*)\]
\[\begin{split}f_{N,1,2n,1}(\cE_1\sub\cE_2\sub\cE_{2n}\sub\cE_{2n+1}\sub\cE_{2n+2},\cE_1\otimes\om^{-1}\inj\cE_2/\cE_1)\\=(\cE_1\sub\cE_{2n+1}\sub\cE_{2n+2},\ev(\cE_1\sub\cE_2,\cE_1\otimes\om^{-1}\inj \cE_2/\cE_1))\end{split}.\] Note that $f_{N,1,2n,1}\circ \pi_{M,1,2n,1}=\tilf_{1,2n,1}$, $\pr_2\circ q_{M,1,2n,1}=\tilc_{1,2n,1}$. We only need to show \[\pr_{2,!}(\pi_{1,2n,1}\times\Delta_{-})^*q_{N,1,2n,1,!}f_{N,1,2n,1}^*(\cF\boxtimes\AS)=0.\] This follows from the following lemma and induction hypothesis:
\begin{lemma}\label{lem:geoctoc2}
    Consider the diagram
    \[\begin{tikzcd}
        \Bun_{P_{1,1,2n-2,1,1}} \ar[r, "p_{1,1,2n-2,1,1}\times q'_{2n-2}\times q''_{2n-2}"] \ar[d, "q_{1,1,2n-2,1,1}"] & \Bun_{P_{1,2n,1}}\times \Bun_{P_{1,1}}\times\Bun_{P_{1,1}} \\
        \Bun_{P_{1,2n-2,1}}
    \end{tikzcd}\] where the maps are \[p_{1,1,2n-2,1,1}(\cE_1\sub\cE_{2}\sub\cE_{2n}\sub\cE_{2n+1}\sub\cE_{2n+2})=(\cE_1\sub\cE_{2n+1}\sub\cE_{2n+2})\]
    \[q'_{2n-2}(\cE_1\sub\cE_{2}\sub\cE_{2n}\sub\cE_{2n+1}\sub\cE_{2n+2})=(\cE_{1}\sub\cE_2)\]
    \[q''_{2n-2}(\cE_1\sub\cE_{2}\sub\cE_{2n}\sub\cE_{2n+1}\sub\cE_{2n+2})=(\cE_{2n+1}/\cE_{2n}\sub\cE_{2n+2}/\cE_{2n})\]
    \[q_{1,1,2n-2,1,1}(\cE_1\sub\cE_{2}\sub\cE_{2n}\sub\cE_{2n+1}\sub\cE_{2n+2})=(\cE_{2}/\cE_1\sub\cE_{2n}/\cE_1\sub\cE_{2n+1}/\cE_1).\] For any $\cF\in\Shv(\Bun_{P_{1,2n,1}})_{\theta-\cusp}$ and $\cG\in\Shv(\Bun_{1,1}\times\Bun_{P_{1,1}})$, we have \[q_{1,1,2n-2,1,1,!}(p_{1,1,2n-2,1,1}\times q'_{2n-2}\times q''_{2n-2})^*(\cF\boxtimes\cG)\in\Shv(\Bun_{P_{1,2n-2,1}})_{\theta-\cusp}.\]
\end{lemma}
The proof of the lemma \ref{lem:geoctoc2} is the same as the proof of \ref{lem:geoctoc1} so we omit.

\end{proof}

Now we come to the proof of theorems \ref{thm:geocuspvan2} and \ref{thm:geocuspvan3}. We have the following immediate consequence of lemma \ref{lem:geopcuspvan}:
\begin{corollary}\label{cor:geocuspvan4}
    For any $\cL\in\Bun_1(K)$, consider the open substack $\Bun_{\Sp_{2n}}^{\bA^{2n}_{\cL}\backslash\{0\}}\sub \Bun_{\Sp_{2n}}^{\bA^{2n}}$ defined as the non-vanishing locus of the map $\alpha:\cL\to\cE_{2n}^s$. Consider \[\pi_{\cL,2n}^{\circ}:\Bun_{\Sp_{2n}}^{\bA^{2n}_{\cL}\backslash\{0\}}\to \Bun_{2n}.\] We have \[\Gamma_c(\pi_{\cL,2n}^{\circ,*}\cF)=0\] for any $\cF\in\Shv(\Bun_{2n})_{\cusp}$.
\end{corollary}
\begin{proof}
    One only needs to note that pull-back along $p_{1,2n,1}:\Bun_{P_{1,2n-2,1}}\to\Bun_{2n}$ sends cuspidal objects to $\theta$-cuspidal objects.
\end{proof}

\begin{proof}[Proof of theorem \ref{thm:geocuspvan2}]
Consider diagram \[\begin{tikzcd}
    \Bun_{\Sp_{2n}}^{\bA^{2n}_{\cL}\backslash\{0\}} \ar[d, "g"] \ar[dr, "\pi_{\cL,2n}^{\circ}"] & \\
    \Bun_{\Sp_{2n}} \ar[r, "\pi_{2n}"] & \Bun_{2n}
\end{tikzcd}.\] We fix a cuspidal object $\cF\in\Shv(\Bun_{2n})$. By the cuspidality, we know that $\cF$ is supported on a quasi-compact open substack of $\Bun_{2n}$ which has bounded Harder-Narasimhan slopes on each connected component. Therefore, for any integer $m\in\mathbb{Z}_{\geq 0}$, one can find $\cL\in\Bun_1$ such that $g$ is a vector bundle of rank $m'\geq m$ with zero section removed over the support of $\pi_{2n}^*\cF$ (the support is denoted $U\sub\Bun_{\Sp_{2n}}$). One has the distinguished triangle
\[\cL_U[-1]\to g_!\uk_{g^{-1}(U)} \to \uk_{U}[-m'] \to \] for some local system $\cL_U\in\Shv(U)^{\heartsuit}$ (we are using the naive t-structure). From this, we get a distinguished triangle \[\Gamma_c(\cL_U\otimes \pi_{2n}^*\cF)[-1] \to \Gamma_c(\pi_{\cL,2n}^{\circ,*}\cF) \to \Gamma_c(\pi_{2n}^*\cF)[-m'] \to .\] By corollary \ref{cor:geocuspvan4}, we get $\Gamma_c(\pi^*_{2n}\cF)[-m']\cong \Gamma_c(\cL_U\otimes \pi_{2n}^*\cF)$. Since the integer $m'$ can be arbitrarily large, we know that $\Gamma_c(\pi_{2n}^*\cF)$ is infinitely connective, hence must be zero. This concludes the proof.
    
\end{proof}

\begin{proof}[Proof of theorem \ref{thm:geocuspvan3}]
    Using stratification $\Bun_{\Sp_{2n}}^{\bA^{2n}_{\cL}}=\Bun_{\Sp_{2n}}^{\bA^{2n}_{\cL}\backslash\{0\}}\cup \Bun_{\Sp_{2n}}$, one can conclude the proof by combining theorem \ref{thm:geocuspvan2} and corollary \ref{cor:geocuspvan4}.
\end{proof}

\phantomsection
\addcontentsline{toc}{part}{References}
\renewcommand{\addcontentsline}[3]{}
\printbibliography

@article{BLX,
      title={The global Gan-Gross-Prasad conjecture for Fourier-Jacobi periods on unitary groups}, 
      author={Boisseau, Paul and Lu, Weixiao and Xue,Hang},
      year={2024},
      eprint={2404.07342},
      archivePrefix={arXiv},
      primaryClass={math.RT}
}

@article{BP21,
    AUTHOR = {Beuzart-Plessis, Rapha\"{e}l},
     TITLE = {Comparison of local relative characters and the
              {I}chino-{I}keda conjecture for unitary groups},
   JOURNAL = {J. Inst. Math. Jussieu},
  FJOURNAL = {Journal of the Institute of Mathematics of Jussieu. JIMJ.
              Journal de l'Institut de Math\'{e}matiques de Jussieu},
    VOLUME = {20},
      YEAR = {2021},
    NUMBER = {6},
     PAGES = {1803--1854},
      ISSN = {1474-7480,1475-3030},
   MRCLASS = {22E50 (11F70 22E55)},
  MRNUMBER = {4332778},
MRREVIEWER = {Baiying\ Liu},
       DOI = {10.1017/S1474748019000707},
       URL = {https://doi.org/10.1017/S1474748019000707},
}

@article{BPCZ,
    AUTHOR = {Beuzart-Plessis, Rapha\"{e}l and Chaudouard, Pierre-Henri and Zydor, Micha\l},
     TITLE = {The global {G}an-{G}ross-{P}rasad conjecture for unitary
              groups: the endoscopic case},
   JOURNAL = {Publ. Math. Inst. Hautes \'{E}tudes Sci.},
  FJOURNAL = {Publications Math\'{e}matiques. Institut de Hautes \'{E}tudes
              Scientifiques},
    VOLUME = {135},
      YEAR = {2022},
     PAGES = {183--336},
      ISSN = {0073-8301,1618-1913},
   MRCLASS = {22E50 (11F70 11R39 22E55)},
  MRNUMBER = {4426741},
MRREVIEWER = {Dongwen\ Liu},
       DOI = {10.1007/s10240-021-00129-1},
       URL = {https://doi.org/10.1007/s10240-021-00129-1},
}

@misc{Zydor19,
      title={Periods of automorphic forms over reductive subgroups}, 
      author={Zydor, Micha\l },
      year={2019},
      eprint={1903.01697},
      archivePrefix={arXiv},
      primaryClass={math.NT},
      url={https://arxiv.org/abs/1903.01697}, 
}

@book{MW95,
    place={Cambridge},
    series={Cambridge Tracts in Mathematics},
    title={Spectral Decomposition and Eisenstein Series: A Paraphrase of the Scriptures},
    DOI={10.1017/CBO9780511470905},
    publisher={Cambridge University Press}, 
    author={Moeglin, C and Waldspurger, J. L.},
    year={1995}
}

@misc{MWZ1,
      title={BZSV Duality for Some Strongly Tempered Spherical Varieties}, 
      author={Zhengyu Mao and Chen Wan and Lei Zhang},
      year={2024},
      eprint={2310.17837},
      archivePrefix={arXiv},
      primaryClass={math.NT},
      url={https://arxiv.org/abs/2310.17837}, 
}

@misc{BZSV,
      title={Relative Langlands Duality}, 
      author={David Ben-Zvi and Yiannis Sakellaridis and Akshay Venkatesh},
      year={2024},
      eprint={2409.04677},
      archivePrefix={arXiv},
      primaryClass={math.RT},
      url={https://arxiv.org/abs/2409.04677}, 
}

@misc{FYZ3,
      title={Modularity of higher theta series I: cohomology of the generic fiber}, 
      author={Tony Feng and Zhiwei Yun and Wei Zhang},
      year={2023},
      eprint={2308.10979},
      archivePrefix={arXiv},
      primaryClass={math.NT},
      url={https://arxiv.org/abs/2308.10979}, 
}

@article{JR,
    author = {Jacquet, Hervé and Rallis, Stephen},
    journal = {Journal für die reine und angewandte Mathematik},
    pages = {175-198},
    title = {Symplectic periods.},
    url = {http://eudml.org/doc/153384},
    volume = {423},
    year = {1992},
}

@article{AGR,
    author={Ash, Avner
    and Ginzburg, David
    and Rallis, Steven},
    title={Vanishing periods of cusp forms over modular symbols},
    journal={Mathematische Annalen},
    year={1993},
    month={Dec},
    day={01},
    volume={296},
    number={1},
    pages={709-723},
    issn={1432-1807},
    doi={10.1007/BF01445131},
    url={https://doi.org/10.1007/BF01445131}
}

@misc{AGKRRV2,
      title={Duality for automorphic sheaves with nilpotent singular support}, 
      author={D. Arinkin and D. Gaitsgory and D. Kazhdan and S. Raskin and N. Rozenblyum and Y. Varshavsky},
      year={2022},
      eprint={2012.07665},
      archivePrefix={arXiv},
      primaryClass={math.AG},
      url={https://arxiv.org/abs/2012.07665}, 
}

@article {Shmelev,
    AUTHOR = {Shmel\"ev, A. S.},
     TITLE = {Orbits of the symplectic group on the manifold of complete
              flags of a symplectic space},
   JOURNAL = {Izv. Ross. Akad. Nauk Ser. Mat.},
  FJOURNAL = {Izvestiya Rossiiskoi Akademii Nauk. Seriya Matematicheskaya},
    VOLUME = {58},
      YEAR = {1994},
    NUMBER = {1},
     PAGES = {144--166},
      ISSN = {1607-0046,2587-5906},
   MRCLASS = {58F06 (22E45 32M10 57R15 58C27)},
  MRNUMBER = {1271518},
MRREVIEWER = {J.\ S.\ Joel},
       DOI = {10.1070/IM1995v044n01ABEH001586},
       URL = {https://doi.org/10.1070/IM1995v044n01ABEH001586},
}

@article {JacquetShalika,
    AUTHOR = {Jacquet, H. and Shalika, J. A.},
     TITLE = {On {E}uler products and the classification of automorphic
              representations. {I}},
   JOURNAL = {Amer. J. Math.},
  FJOURNAL = {American Journal of Mathematics},
    VOLUME = {103},
      YEAR = {1981},
    NUMBER = {3},
     PAGES = {499--558},
      ISSN = {0002-9327,1080-6377},
   MRCLASS = {10D40 (12A67 22E55)},
  MRNUMBER = {618323},
MRREVIEWER = {Freydoon\ Shahidi},
       DOI = {10.2307/2374103},
       URL = {https://doi.org/10.2307/2374103},
}

@article {LinearPeriods,
    AUTHOR = {Friedberg, Solomon and Jacquet, Herv\'e},
     TITLE = {Linear periods},
   JOURNAL = {J. Reine Angew. Math.},
  FJOURNAL = {Journal f\"ur die Reine und Angewandte Mathematik. [Crelle's
              Journal]},
    VOLUME = {443},
      YEAR = {1993},
     PAGES = {91--139},
      ISSN = {0075-4102,1435-5345},
   MRCLASS = {11F70 (22E55)},
  MRNUMBER = {1241129},
MRREVIEWER = {Stephen\ Gelbart},
       DOI = {10.1515/crll.1993.443.91},
       URL = {https://doi.org/10.1515/crll.1993.443.91},
}

@article{Arthur78,
  title={A trace formula for reductive groups I terms associated to classes in G (Q)},
  author={Arthur, James G},
  year={1978}
}

@book{Langlands76,
  title={On the functional equations satisfied by Eisenstein series},
  author={Langlands, Robert P},
  volume={544},
  year={2006},
  publisher={Springer}
}

@article{BL24,
  title={On the meromorphic continuation of Eisenstein series},
  author={Bernstein, Joseph and Lapid, Erez},
  journal={Journal of the American Mathematical Society},
  volume={37},
  number={1},
  pages={187--234},
  year={2024}
}

@book{Lafforgue,
  title={Chtoucas de Drinfeld et conjecture de Ramanujan-Petersson},
  author={Lafforgue, Laurent},
  year={1997},
  publisher={Soci{\'e}t{\'e} math{\'e}matique de France}
}

@article{Gross97,
  title={On the motive of a reductive group},
  author={Gross, Benedict H},
  journal={Inventiones mathematicae},
  volume={130},
  number={2},
  pages={287--314},
  year={1997},
  publisher={Berlin, Springer-Verlag.}
}

@misc{liu2017enhancedadicformalismperverse,
      title={Enhanced adic formalism and perverse t-structures for higher Artin stacks}, 
      author={Yifeng Liu and Weizhe Zheng},
      year={2017},
      eprint={1404.1128},
      archivePrefix={arXiv},
      primaryClass={math.AG},
      url={https://arxiv.org/abs/1404.1128}, 
}

@misc{arinkin2022stacklocalsystemsrestricted,
      title={The stack of local systems with restricted variation and geometric Langlands theory with nilpotent singular support}, 
      author={D. Arinkin and D. Gaitsgory and D. Kazhdan and S. Raskin and N. Rozenblyum and Y. Varshavsky},
      year={2022},
      eprint={2010.01906},
      archivePrefix={arXiv},
      primaryClass={math.AG},
      url={https://arxiv.org/abs/2010.01906}, 
}

@misc{LW,
      title={Higher Period Integrals and Derivatives of L-functions}, 
      author={Shurui Liu and Zeyu Wang},
      year={2025},
      eprint={2504.00275},
      archivePrefix={arXiv},
      primaryClass={math.NT},
      url={https://arxiv.org/abs/2504.00275}, 
}

@article{lysenko2002local,
  title={Local geometrised Rankin-Selberg method for GL(n)},
  author={Lysenko, Sergey},
  journal={DUKE MATHEMATICAL JOURNAL},
  volume={111},
  number={3},
  year={2002}
}

@article{lysenko2008geometric,
  title={Geometric Waldspurger periods},
  author={Lysenko, Sergey},
  journal={Compositio Mathematica},
  volume={144},
  number={2},
  pages={377--438},
  year={2008},
  publisher={London Mathematical Society}
}

@article{Lysenko_2021,
   title={Linear Periods of Automorphic Sheaves for GL2n},
   volume={2022},
   ISSN={1687-0247},
   url={http://dx.doi.org/10.1093/imrn/rnab077},
   DOI={10.1093/imrn/rnab077},
   number={17},
   journal={International Mathematics Research Notices},
   publisher={Oxford University Press (OUP)},
   author={Lysenko, S},
   year={2021},
   month=may, pages={12984–13053} }

@article{JR92,
  title={Symplectic periods},
  author={Jacquet, Herv{\'e} and Rallis, Stephen},
  journal={J. reine angew. Math},
  volume={423},
  number={1992},
  pages={175--197},
  year={1992}
}

@article{Offen06,
  title={Residual spectrum of GL 2 n distinguished by the symplectic group},
  author={Offen, Omer},
  year={2006}
}

@article{MO21,
  title={On distinguished representations of the quasi-split unitary groups},
  author={Mitra, Arnab and Offen, Omer},
  journal={Journal of the Institute of Mathematics of Jussieu},
  volume={20},
  number={1},
  pages={225--276},
  year={2021},
  publisher={Cambridge University Press}
}

\end{document}